\documentclass[a4paper,11pt,pdf]{amsart}
\usepackage{enumerate, amsmath, amsfonts, amssymb, amsthm, thmtools, wasysym, graphics, graphicx, xcolor, frcursive,comment,bbm}

\usepackage{etex}

\definecolor{darkblue}{rgb}{0.0,0,0.7} 
 
\definecolor{darkred}{rgb}{0.7,0,0} 
\usepackage{hyperref}
\usepackage[all]{xy}
\usepackage[T1]{fontenc}

\usepackage{vmargin}            
\setmarginsrb{3.3cm}{3.2cm}{3.3cm}{3.1cm}{0cm}{0.6cm}{0cm}{0cm}       

\usepackage{caption,lipsum}
\usepackage{graphicx}                  
\usepackage{pstricks,pst-plot,pst-text,pst-tree,pst-eps,pst-fill,pst-node,pst-math}
\usepackage{setspace}
\usepackage{multicol}

\newcommand{\darkred}{\color{darkred}} 
\newcommand{\defn}[1]{\emph{\darkred #1}}

\def\S{{\mathcal{S}}}
\def\R{{\mathcal{R}}}
\def\T{{\mathcal{T}}}

\def\Div{{\sf{Div}}}

\def\S{{\mathcal{S}}}

\usepackage{etex}

\newtheorem{theorem}{Theorem}[section]
\newtheorem{prop}[theorem]{Proposition}
\newtheorem{lemma}[theorem]{Lemma}
\newtheorem{cor}[theorem]{Corollary}
\newtheorem{conjecture}[theorem]{Conjecture}
\newtheorem{nota}[theorem]{Notation}

\theoremstyle{definition}
\newtheorem{definition}[theorem]{Definition}
\newtheorem{rmq}[theorem]{Remark}
\newtheorem{exple}[theorem]{Example}
\newtheorem{question}[theorem]{Question}

\numberwithin{equation}{section}

\title[On torus knot groups and a submonoid of the braid group]{On torus knot groups and a submonoid of the braid group}

\author{Thomas Gobet}
\address{Institut Denis Poisson, CNRS UMR 7350, Facult\'e des Sciences et Techniques, Universit\'e de Tours, Parc de Grandmont, 
37200 TOURS, France}
\email{thomas.gobet@lmpt.univ-tours.fr}

\begin{document}
\maketitle

\bigskip

\begin{center}
\textit{Dedicated to the memory of Patrick Dehornoy.
}\end{center}

\bigskip

\begin{abstract}
The submonoid of the $3$-strand braid group~$\mathcal{B}_3$ generated by $\sigma_1$ and $\sigma_1 \sigma_2$ is known to yield an exotic Garside structure on $\mathcal{B}_3$. We introduce and study an infinite family~$(M_n)_{n\geq 1}$ of Garside monoids generalizing this exotic Garside structure, \emph{i.e.}, such that $M_2$ is isomorphic to the above monoid. The corresponding Garside group~$G(M_n)$ is isomorphic to the~$(n,n+1)$-torus knot group--which is isomorphic to~$\mathcal{B}_3$ for $n=2$ and to the braid group of the exceptional complex reflection group~$G_{12}$ for $n=3$. This yields a new Garside structure on $(n,n+1)$-torus knot groups, which already admit several distinct Garside structures. 

The $(n,n+1)$-torus knot group is an extension of $\mathcal{B}_{n+1}$, and the Garside monoid~$M_n$ surjects onto the submonoid~$\Sigma_n$ of $\mathcal{B}_{n+1}$ generated by $\sigma_1, \sigma_1 \sigma_2, \dots, \sigma_1 \sigma_2\cdots \sigma_n$, which is \textit{not} a Garside monoid when $n>2$. Using a new presentation of $\mathcal{B}_{n+1}$ that is similar to the presentation of $G(M_n)$, we nevertheless check that $\Sigma_n$ is an Ore monoid with group of fractions isomorphic to $\mathcal{B}_{n+1}$, and give a conjectural presentation of it, similar to the defining presentation of $M_n$. This partially answers a question of Dehornoy--Digne--Godelle--Krammer--Michel.  
\end{abstract}

\tableofcontents

\section{Introduction}

\thispagestyle{empty}

The braid group on $n$ strands is one of the most basic example of a Garside group. Garside groups, originally introduced by~Dehornoy and Paris~\cite{DP} following an original idea of Garside~\cite{garside_69}, are defined as groups of fractions of certain monoids, called \textit{Garside monoids}, which have enough properties to ensure that every element of the group can be written uniquely as an irreducible fraction in two elements of the monoid. Computable normal forms for elements of these monoids can be defined, allowing one to effectively compute such fractions, which in particular yields a solution to the word problem in these groups. Garside groups also have many other properties. For example, they are torsion-free, and have a solvable conjugacy problem--see Section~\ref{sec:garside} below for basic definitions and properties of Garside monoids and groups, and~\cite{Garside} for more on the topic.  

While the word problem in the $n$-strand braid group has been known to be solvable since Artin's original paper~\cite{Artin} and several other approaches have been shown to be fruitful in between (see~\cite[Section~5]{BB} for a survey), Garside's approach allowed him to get the first solution to the conjugacy problem, and his results were generalized to get a uniform solution to these questions in Artin--Tits groups of spherical type~\cite{BS, Del}, \emph{i.e.}, Artin--Tits groups attached to finite Coxeter groups (see~\cite[Section~6.6]{KT} for an introduction to the topic). It also provides new proofs that Artin--Tits groups of spherical type are torsion-free, and allows one to determine their center. One can also note that Garside normal forms can be used to show faithfulness of (linear, and more recently categorical) representations of Garside groups~\cite{Krammer, BT, Jensen, LQ}. Roughly speaking, Garside groups are groups satisfying a set of axioms that ensures that generalizations of the techniques of Garside can be applied to solve the above-mentioned problems. 

In general, the Garside group does not determine an associated Garside monoid, \emph{i.e.}, several non-isomorphic Garside monoids may have isomorphic group of fractions (see~\cite[Section~6.4, Problem~10]{Dual}). Up to now, it seems that very few classification results of Garside monoids for a given Garside group are known. In the case of the $n$-strand braid group, Garside's original paper yields a so-called \textit{classical} Garside monoid, which is nothing but the positive braid monoid, while Birman, Ko, and~Lee~\cite{BKL} discovered a second Garside monoid, which strictly contains the first one. This Garside monoid is generated by a copy of the set of transpositions of the symmetric group. Bessis, Digne, and~Michel~\cite{BDM} generalized this monoid to Artin--Tits groups of Coxeter type $B_n$, and then Bessis gave a generalization of these constructions, called \textit{dual braid monoid}, which is valid for every Artin--Tits group attached to a finite Coxeter system~\cite{Dual}, and even to braid groups of well-generated complex reflection groups~\cite{Dualcplx}. Following Bessis' approach, some Artin--Tits groups of non-spherical type were also shown to be (quasi-)Garside groups~\cite{Dig, Dig1, Bessis_free}.  

Birman and Brendle asked if there exist other Garside monoids for the $n$-strand braid group (see~\cite[Open Problem~10]{BB}, where it is also claimed that it is very likely that the classical and dual presentations of $\mathcal{B}_{n+1}$ are the only presentations yielding a Garside monoid). There are several motivations for looking for other Garside presentations of $\mathcal{B}_n$. In addition to classification perspectives, one can cite for instance the look for a polynomial algorithm for the conjugacy problem. At the time of writing of this paper, it seems that the only known Garside monoids which can be defined for the $n$-strand braid group for all $n\geq 1$ are still the classical and the dual braid monoids. 

Nevertheless, for $n=3$, several exotic Garside monoids for the $3$-strand braid group~$\mathcal{B}_3$ were discovered (see~\cite[Section~IX.2.4]{Garside} for a survey). Two of them are given by the following presentations : 
\[\langle~ x, y \ \vert\ 
\begin{matrix}
x^2=y^3
\end{matrix}
\ \rangle\text{~~and~~} \langle~ a, b \ \vert\ 
\begin{matrix}
aba=b^2
\end{matrix}
\ \rangle\]
It is natural to wonder whether these monoids admit analogues in higher rank or if they should be considered as some sort of sporadic monoids only arising in low rank. For the first one, one can answer this question as follows: this presentation is in fact a presentation of the~\emph{torus knot group} of the torus knot~$T_{2,3}$: given $n,m$ two relatively prime integers, the torus knot group $G(n,m)$ is the fundamental group of the complement of the torus knot $T_{n,m}$. It has a presentation with two generators $x,y$ and a single relation $x^n=y^m$, and this presentation is known to yield a Garside monoid (see~\cite[Example~4]{DP}). One has an isomorphism $\mathcal{B}_3\cong G(2,3)$, while in general for $m=n+1$ one only has a surjection $G(n,n+1)\twoheadrightarrow \mathcal{B}_{n+1}$. Note that several other Garside structures for $G(n,m)$ are known (see Section~\ref{garside_torus} below). In this paper, we investigate the question for the second above-mentioned exotic Garside structure on $\mathcal{B}_3$. In terms of the classical generators, one has $a=\sigma_1$, $b=\sigma_1 \sigma_2$. It was the first example of a Garside monoid where the lcm of the atoms is not equal to the Garside element (see~\cite[Exemple 1.5]{dehornoy_garside}). Indeed, in this Garside monoid, the left-lcm of $a$ and $b$ is $b^2$, while the Garside element $\Delta$ is $b^3$ (the lattice of divisors of $\Delta$ under left-divisibility is given in Figure~\ref{fig_intro}). In fact, in the original paper~\cite{DP}, it was a requirement for the Garside element $\Delta$ to be the lcm of the atoms, but this condition was slightly relaxed in~\cite{dehornoy_garside}, and is not required anymore in the definition of Garside monoid which is used nowadays.

\begin{figure}
\begin{pspicture}(-2,-0.3)(3.5,5)
\psdots(1.5,0)(0.5,0.7)(2.5,1.4)(0.5,2.1)(2.5,2.1)(0.5,2.8)(2.5,3.5)(1.5,4.2)
\psline(1.5,0)(0.5,0.7)
\psline(0.5,0.7)(0.5,2.8)
\psline(1.5,0)(2.5,1.4)
\psline(2.5,1.4)(0.5,2.8)
\psline(2.5,1.4)(2.5,3.5)
\psline(2.5,3.5)(1.5,4.2)
\psline(0.5,2.8)(1.5,4.2)
\rput(1,-0.1){$1$}
\rput(0.15,0.7){$a$}
\rput(0,2.1){$ab$}
\rput(-0.5,2.8){$b^2=aba$}
\rput(-0.25,4.2){$b^3=abab=baba$}
\rput(2.9,1.4){$b$}
\rput(3,2.1){$ba$}
\rput(3.1,3.5){$bab$}
\end{pspicture}
\caption{The lattice of simples in the submonoid of $\mathcal{B}_3$ generated by $a=\sigma_1$ and $b=\sigma_1\sigma_2$.}
\label{fig_intro}
\end{figure}
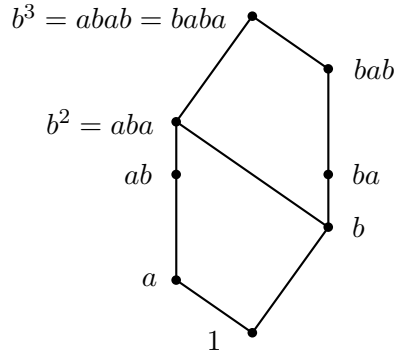    

The submonoid of $\mathcal{B}_3$ mentioned above admits a natural generalization to $\mathcal{B}_{n+1}$, $n\geq 2$, given by the submonoid $\Sigma_{n}\subseteq \mathcal{B}_{n+1}$ generated by $\sigma_1, \sigma_1 \sigma_2, \dots, \sigma_1 \sigma_2 \cdots \sigma_n$. In~\cite[Chapter~IX, Question~30]{Garside}, the following question is raised:

\begin{question}\label{q1}
Does the submonoid $\Sigma_n$ admit a finite presentation ? Is it a Garside monoid ?
\end{question}

A positive answer to the last question would in particular yield a new Garside structure on $\mathcal{B}_{n+1}$, generalizing the exotic Garside structure given by $\Sigma_2$ on $\mathcal{B}_3$. Unfortunately, the submonoid $\Sigma_n$ is \textit{not} a Garside monoid when $n>3$: in fact, as already noticed by~Dehornoy before Question~\ref{q1} was asked, this monoid does not have lcm's (as follows easily from~\cite[Example~3.7]{subword}: there it is shown that a monoid conjecturally isomorphic to the opposite monoid of $\Sigma_3$ does not have lcm's, and the same argument can be given for $\Sigma_3$). But we shall show that the $(n,n+1)$-torus knot group $G(n,n+1)$ admits a Garside structure generalizing the above mentioned exotic Garside structure, and having as image the submonoid $\Sigma_{n}$. In other words, the above-mentioned exotic Garside monoid admits a generalization $M_n$, which has as group of fractions an extension of the braid group, isomorphic to it in low ranks.  

Let us now define our main object of study. Let $n\geq 1$ and let $M_n$ be the monoid defined by the presentation \[\langle~ \rho_1, \rho_2, \dots, \rho_n \ \vert\ 
\begin{matrix}
\rho_1 \rho_n \rho_i=\rho_{i+1} \rho_n\text{~for~}1\leq i \leq n-1
\end{matrix} \ \rangle\] Then our main results can be summarized as follows (see Theorem~\ref{garside_gn}, Propositions~\ref{prop:surjective} and \ref{prop:iso}, and Corollary~\ref{lcm_s} below)

\begin{theorem}\label{main_intro}
We have \begin{enumerate} \item The monoid $M_n$ is a Garside monoid, with (central) Garside element $\Delta=\rho_n^{n+1}$, and (left- or right-) lcm of the atoms $\rho_n^n$. 
\item The Garside group $G(M_n)$ obtained as group of fractions of $M_n$ is isomorphic to the $(n,n+1)$-torus knot group. In particular for $n=1$ and $n=2$ we have $G(M_n) \cong \mathcal{B}_{n+1}$, while for $n>2$ it is a proper extension of $\mathcal{B}_{n+1}$. 
\item The image of $M_n$ in $\mathcal{B}_{n+1}$ under the above-mentioned surjection onto $\mathcal{B}_{n+1}$ is the submonoid $\Sigma_n$. In particular $M_2=\Sigma_2$ holds.  
\end{enumerate}
\end{theorem}

The center of $G(M_n)$ (which is known to be infinite cyclic) is generated by the Garside element $\Delta=\rho_n^{n+1}$ for~$n\geq 2$ and by $\rho_1$ for~$n=1$.

It is known that the $(3,4)$-torus knot group is isomorphic to the braid group of the complex reflection group $G_{12}$. Irreducible complex reflection groups which are well-generated admit a so-called \textit{dual braid monoid} by work of Bessis~\cite{Dualcplx}. The complex reflection group $G_{12}$ is not well-generated but since it is isomorphic to the $(3,4)$-torus knot group, it admits several Garside structures, including, in some sense, a classical and a dual one (see Section~\ref{garside_torus} below). The above theorem specialized at $n=3$ yields an additional Garside structure for its braid group (and a new presentation of $G_{12}$ can be derived). Note that, at the time of writing, the only irreducible complex reflection group for which is it not known whether the corresponding braid group is a Garside group or not is $G_{31}$ (see Remark~\ref{rmq:cplx_gps} below). 

\medskip

Coming back to Question~\ref{q1}, one can define a presentation of the braid group~$\mathcal{B}_{n+1}$ which is closely related to that of $M_n$. Let $\mathcal{H}_n^+$ be the quotient of $M_n$ defined by the presentation \begin{equation}\label{hn} \langle~ \rho_1, \rho_2, \dots, \rho_n \ \vert\ \rho_1 \rho_j \rho_i=\rho_{i+1} \rho_j\text{~for~}1\leq i < j \leq n~ \rangle\end{equation} 
Then we show (see Propositions~\ref{pres_bn} and \ref{prop_ore})

\begin{prop}
We have
\begin{enumerate}
\item The submonoid $\Sigma_n$ of $\mathcal{B}_{n+1}$ is an Ore monoid with group of fractions isomorphic to $\mathcal{B}_{n+1}$.
\item The group with presentation~\ref{hn} is isomorphic to $\mathcal{B}_{n+1}$ via $\rho_i\mapsto \sigma_1 \sigma_2\cdots \sigma_i$. The image of $\mathcal{H}_{n}^+$ inside $\mathcal{B}_{n+1}$ is $\Sigma_n$. 
\end{enumerate}
\end{prop}

We then conjecture the following (see Conjecture~\ref{conj} below for a more precise statement)

\begin{conjecture}
The monoid $\mathcal{H}_n^+$ is cancellative. As a corollary, we have $\mathcal{H}_n^+\cong \Sigma_n$, and $\Sigma_n$ admits a finite presentation. 
\end{conjecture}

This would positively answer the first part of Question~\ref{q1}. Note that in the particular case $n=3$, Dehornoy asked whether $\mathcal{H}_3^+$ is (right-)cancellative and embeds into its group of fractions (see~\cite[Question~3.8]{subword}--note that the monoid defined there is the opposite monoid of~$\mathcal{H}_3^+$).  

\bigskip

The paper is organized as follows: Section~\ref{sec:garside} is devoted to recalling definitions and properties of Garside monoids and groups, and collecting a few general results which are used later on. In Section~\ref{garside_torus} we recall some basic facts about torus knot groups and their Garside structures. In Section~\ref{sec:main} we introduce the monoids~$M_n$, give several presentations of them, and show that they are Garside monoids (with Garside group~$G(M_n)$ isomorphic to the $(n,n+1)$-torus knot group) using the so-called \textit{reversing approach}. In Section~\ref{sec:link} we explore the link between $G(M_n)$ and $\mathcal{B}_{n+1}$ and give a few properties as well as a conjectural presentation of the submonoid $\Sigma_n$ of $\mathcal{B}_{n+1}$. Section~\ref{sec:dihedral} is devoted to showing that Artin--Tits groups of odd dihedral type can be endowed with a Garside structure that is analogous to the one given by $M_n$.   

\smallskip

{\bf Acknowledgements.} The author thanks Ivan Marin, Jean Michel, Matthieu Picantin, and Baptiste Rognerud for useful discussions. He also thanks an anonymous referee for pointing out that the groups studied in the paper were in fact torus knot groups, and for many relevant comments and suggestions.

\section{Garside monoids and groups}\label{sec:garside}

The aim of this section is to recall a few basic results on Garside monoids and Garside goups for later use. We mostly adopt the definitions and conventions from~\cite{Garside}. Note that, while \textit{loc. cit.} introduces most of the results used in this paper in the general framework of \textit{Garside categories}, we will only need them in the case of presented monoids, and therefore reproduce them here in this less general context for the comfort of the reader. We also include proofs of a few basic results. 

\subsection{Definitions and properties}

Every monoid has a unit element $1$. Let $M$ be a monoid. 

\begin{definition}[Divisors and multiples]
Let $a,b,c\in M$. If $ab=c$ holds, we say that $a$ is a \defn{left-divisor} (respectively, that $b$ is a \defn{right-divisor}) of $c$ and that $c$ is a \defn{right-multiple} of $a$ (respectively a \defn{left-multiple} of $b$).
\end{definition}

\begin{definition}[Cancellativity]
We say that $M$ is \defn{left-cancellative} (respectively \defn{right-cancellative}) if for all $a,b,c\in M$, the equality $ab=ac$ (resp. $ba=ca$) implies $b=c$. If~$M$ is both left- and right-cancellative then we simply say that $M$ is \defn{cancellative}. 
\end{definition}

\begin{theorem}[Ore's Theorem]\label{thm:ore}
If $M$ is cancellative, and if any two elements $a,b\in M$ admit a common left-multiple, that is, if there is $c\in M$ satisfying $a'a=c=b'b$ for some $a',b'\in M$, then $M$ admits a group of fractions $G(M)$ in which it embeds. Moreover, if $\langle \S, \R\rangle$ is a presentation of the monoid $M$, then $\langle \S, \R\rangle$ is a presentation of $G(M)$. 
\end{theorem}

A proof of this Theorem can be found for instance in~\cite[Section~1.10]{CP}.

\begin{definition}[Ore monoid]
A monoid satisfying the assumptions of Theorem~\ref{thm:ore} is an \defn{Ore monoid}. 
\end{definition}

\begin{lemma}\label{lem_partial}
If $M$ is left-cancellative (respectively right-cancellative) and $1$ is the only invertible element in $M$, then the left-divisibility (resp. right-divisibility) relation on~$M$ is a partial order. 
\end{lemma}

\begin{proof}
Reflexivity is clear as $M$ has a unit $1$ and transitivity is also clear (and both hold without the cancellativity assumption and without the assumption on invertible elements). Let $a,b\in M$ such that $a$ left-divides $b$ and $b$ left-divides $a$. Then there are $c,c'\in M$ satisfying $ac=b$ and $bc'=a$. Hence we get $b=ac=bc'c$. By left-cancellativity this implies $c'c=1$, hence $c=1=c'$ as $1$ is the only invertible element in $M$. Hence $a=b$, and the left-divisibility relation is reflexive. The proof of the right counterparts is similar. 
\end{proof}

\begin{definition}[Noetherian divisibility]\label{def:noeth}
We say that the divisibility in $M$ is \defn{Noetherian} if there exists a function $\lambda: M \rightarrow \mathbb{Z}_{\geq 0}$ satisfying $\forall a,b\in M$, $\lambda(ab)\geq \lambda(a)+\lambda(b)$ and $a\neq 1 \Rightarrow \lambda(a)\neq 0$. We say that $M$ is \defn{right-Noetherian} (respectively \defn{left-Noetherian}) if every strictly increasing sequence of divisors with respect to left-divisibility (resp. right-divisibility) is finite. Note that if the divisibility in $M$ is Noetherian, then $M$ is both left- and right-Noetherian.
\end{definition}

Note that it implies that the only invertible element in $M$ is $1$ and that $M$ is infinite for~$M\neq \{1\}$. In particular, by Lemma~\ref{lem_partial}, in a cancellative monoid $M$ with Noetherian divisibility, both left-divisiblity and right-divisibility induce a partial order on~$M$.

\begin{definition}[Garside monoid]\label{def_garside}
A \defn{Garside monoid} is a pair~$(M, \Delta)$ where $M$ is a monoid and $\Delta$ is an element of $M$, satisfying the following five conditions
\begin{enumerate}
\item $M$ is left- and right-cancellative,
\item the divisibility in $M$ is Noetherian,
\item any two elements in $M$ admit a left- and right-lcm, and a left- and right-gcd,
\item the left- and right-divisors of the element $\Delta$ coincide and generate $M$,
\item the set of (left- or right-)divisors of $\Delta$ is finite.
\end{enumerate}
\end{definition}

Note that under these assumptions, the restrictions of left- and right-divisibility to the set of divisors of $\Delta$ yield two lattice structures on this set. 

In general, checking the above five conditions is a nontrivial task, especially for the left- and right-cancellativity. But these conditions have strong implications. We list some of them below, and refer the reader to~\cite{Garside} for complete proofs. 

Let $M$ be a Garside monoid. Firstly, by Ore's Theorem, we have that $M$ embeds into its group of fractions $G(M)$.

\begin{definition}[Garside group]
A group $G$ is a \defn{Garside group} if $G=G(M)$ holds for some Garside monoid $M$. 
\end{definition} 

Secondly, one can define normal forms for elements of~$M$ as products of divisors of the Garside element: let $a\in M$. As $M$ has gcd's, let $x_1=\mathrm{gcd}(a, \Delta)$ (we consider left-gcd's here). Hence $a=x_1 y_1$, and $x_1$ is the greatest divisor of $\Delta$ which also left-divides $a$. By cancellativity, the element $y_1$ is uniquely determined, and one can go on, considering the greatest left-divisor $x_2$ of $y_1$ which also divides $\Delta$. We then write $a=x_1 x_2 y_2$. In this way, we get a uniquely defined sequence of divisors of $\Delta$, and as the divisibility is Noetherian in $M$, this sequence is finite. At the end we get a uniquely defined expression $a=x_1 x_2 \cdots x_k$ as product of divisors of $\Delta$. This normal form is called the \defn{(left-)Garside normal form} of $a$. It can be effectively computed provided that left-gcd's of the form $\mathrm{gcd}(xy, \Delta)$, where $x$ and $y$ are divisors of $\Delta$, can be computed. Indeed, one can show that for~$x,y\in M$, one has $\mathrm{gcd}(xy, \Delta)=\mathrm{gcd}(x(\mathrm{gcd}(y,\Delta)), \Delta)$; this allows one to compute the normal from of $a$ starting from any expression of $a$ as a product of divisors of $\Delta$ (which generate $M$). Namely if $a$ is equal to $a_1 a_2\cdots a_k$ with $a_i$ dividing $\Delta$ for all $i$, then an iterated application of the above formula reduces the computation of the first factor of the Garside normal form to an iterated computation of gcd's of the above form. Similarly, one can define a \textit{right-Garside normal form}.    

Thirdly, the important point about (left-)normal forms in $M$ is that they can be used, in the case where $M$ and $G(M)$ are defined by generators and relations, to give a solution to the word problem in $G(M)$. We say that the word problem in a (finitely generated) group $G$ is \defn{solvable} if there is an algorithm which allows one to determine in finite time whether a word in the generating set represents the identity or not. If $G(M)$ is a Garside group, then it can be checked that every element of $G(M)$ can be written uniquely as an irreducible fraction $x^{-1} y$ with $x, y\in M$, which can be computed using the left-normal form in $M$. The normal form can also be used to give a solution to the conjugacy problem in Garside groups. 

Finally, it can also be shown that every Garside group $G(M)$ is torsion-free, and that a power of $\Delta$ is central in $G(M)$--hence in particular, that the center of $G(M)$ is nontrivial.  

In Sections~\ref{cancel} and~\ref{garside_elts}, we will recall a few existing tools for checking some of the conditions of Definition~\ref{def_garside} in the case of presented monoids.   

\begin{exple}\label{ex_1}
The seminal example is given by braid groups, or more generally Artin--Tits groups of spherical type (\emph{i.e.}, attached to a finite Coxeter system). Let $n\geq 1$. Recall that the~$(n+1)$-strand braid group~$\mathcal{B}_{n+1}$ has a presentation
\[\bigg\langle~ \sigma_1, \sigma_2, \dots, \sigma_n \ \bigg\vert\ \begin{matrix} \sigma_i \sigma_{i+1} \sigma_i=\sigma_{i+1} \sigma_i \sigma_{i+1}~\text{for~}1\leq i<n,\\ \sigma_i \sigma_j = \sigma_j \sigma_i ~\text{for~}|i-j|>1.\end{matrix}~\bigg\rangle\]
A possible Garside monoid $M$ satisfying $G(M)\cong \mathcal{B}_{n+1}$ is given by the positive braid monoid~$\mathcal{B}_{n+1}^+$ defined by the same presentation (but as monoid) as the one given above. The element~$\Delta$ is given by the half-twist, \emph{i.e.}, the lift of the longest permutation of~$\mathfrak{S}_{n+1}$ in $\mathcal{B}_{n+1}^+$. This is the \defn{classical} Garside structure on~$\mathcal{B}_{n+1}$. An alternative Garside monoid~$M'$ such that $G(M')\cong \mathcal{B}_{n+1}$ is given by the \textit{Birman--Ko--Lee braid monoid}~\cite{BKL} (or \textit{dual braid monoid}~\cite{Dual}). In this case, the monoid~$M'$ contains $M$, and the element~$\Delta$ is given by $\sigma_1 \sigma_2\cdots \sigma_{n}$. Both the classical and dual Garside structures generalize to Artin--Tits groups of spherical type, leading to two distinct and uniform solutions to the word problem in these groups.  
\end{exple}

\begin{exple}\label{ex_2}
The two Garside structures (classical and dual) given in Example~\ref{ex_1} are the only known Garside structures on $\mathcal{B}_{n+1}$ which can be defined for all~$n\geq 1$. Whether there exist other Garside structures that can be defined for all~$n\geq 1$ or not is an open problem. For~$n=2$, a few exotic Garside structures are known (see~\cite[Section~X.2.4]{Garside}). In this case, the classical braid monoid $\mathcal{B}_3^+$ has generators $\sigma_1, \sigma_2$ and element $\Delta$ given by $\sigma_1 \sigma_2 \sigma_1= \sigma_2 \sigma_1 \sigma_2=\Delta$ (the half-twist). The dual braid monoid $\mathcal{B}_3^*$ has generators $\sigma_1, \sigma_2, \sigma_1 \sigma_2 \sigma_1^{-1}$ and element $\Delta$ given by $\sigma_1 \sigma_2$. An exotic Garside monoid is given by the submonoid $\Sigma_2$ with generators $\rho_1=\sigma_1, \rho_2=\sigma_1 \sigma_2$ and element $\Delta$ given by $(\sigma_1 \sigma_2 \sigma_1)^2=\rho_2^3$. A presentation of $\Sigma_2$ is given by the single relation $\rho_1 \rho_2 \rho_1=\rho_2^2$. Another exotic Garside monoid for $\mathcal{B}_3$ is given by the monoid with generators $x,y$ and a single relation $x^2=y^3$: in fact, this is a presentation of the knot group of the trefoil knot (which is a torus knot); by~\cite[Example~4]{DP}, torus knot groups are known to be Garside groups. In terms of the classical generators we have $x=\sigma_1^2 \sigma_2$, $y=\sigma_1\sigma_2$.     
\end{exple}

\subsection{Cancellativity criteria for presented monoids}\label{cancel}

This section is devoted on recalling some known cancellativity criteria for presented monoids which will be used in Section~\ref{sec:main}. We recall them from~\cite[Section~II.4]{Garside} (extending approaches from~\cite{dehornoy_garside}; see also~\cite{dehornoy_monoids} for more recent results). Most of the definitions given in this section are also borrowed from~\cite{Garside}. 

Assume that $M$ is a monoid defined by a presentation $\langle \S, \R \rangle$, where $\S$ is a finite set of generators and $\R$ a set of relations between words in $\S^*$, \emph{i.e.}, words with letters in the generating set $\S$.

\begin{definition}[Right-complemented presentation]
The presentation $\langle\mathcal{S}, \mathcal{R}\rangle$ is \defn{right-complemented} if $\R$ contains no relation where one side is equal to the empty word, no relation of the form $s\cdots = s\cdots$ with $s\in\mathcal{S}$, and if for $s\neq t\in\mathcal{S}$, there is at most one relation of the form $s\cdots = t\cdots$ in $\mathcal{R}$. 
\end{definition}

\begin{exple}
The classical presentation of the $(n+1)$-strand braid group that we recalled in Example~\ref{ex_1} is right-complemented. More generally, the standard presentation of any Artin--Tits group is right-complemented. 
\end{exple}

Given a right-complemented presentation~$\langle\mathcal{S}, \mathcal{R}\rangle$ of a monoid~$M$, there is a uniquely determined partial map~$\theta: \mathcal{S}\times \mathcal{S}\longrightarrow \mathcal{S}^*$ such that $\theta(s, s)=1$ holds for all $s\in\mathcal{S}$ and such that for $s\neq t\in\mathcal{S}$, the words $\theta(s,t)$ and $\theta(t,s)$ are defined whenever there is a relation $s\cdots = t\cdots$ in $\mathcal{R}$, and are such that this relation is given by $s \theta(s,t) = t \theta(t,s)$. The map~$\theta$ is the \defn{syntactic right-complement} attached to the right-complemented presentation~$\langle\mathcal{S}, \mathcal{R}\rangle$. 

If $\langle\S, \R\rangle$ is right-complemented, then by~\cite[Lemma~II.4.6]{Garside}, the map~$\theta$ admits a unique minimal extension to a partial map from~$\mathcal{S}^* \times \mathcal{S}^*$ to~$\mathcal{S}^*$ which we still denote $\theta$, and satisfying \begin{eqnarray}
\label{c}\theta(s,s)=1,~\forall s\in\mathcal{S},\\
\theta(ab,c)=\theta(b, \theta(a,c)),~\forall a,b,c\in\mathcal{S}^*, \\
\theta(a,bc)=\theta(a,b)\theta(\theta(b,a),c),~\forall a,b,c\in\mathcal{S}^*,\\
\label{d}\theta(1,a)=a\text{ and }\theta(a,1)=1,~\forall a\in\mathcal{S}^*.
\end{eqnarray}

\begin{definition}[Cube condition]\label{def_cube} Given a right-complemented presentation~$\langle\mathcal{S}, \mathcal{R}\rangle$ of a monoid~$M$ with syntactic right-complement~$\theta$, we say that the \defn{$\theta$-cube condition holds} (respectively that the \defn{sharp $\theta$-cube condition holds}) for a triple~$(a,b,c)\in({\mathcal{S}^*})^3$ if either both $\theta(\theta(a,b), \theta(a,c))$ and $\theta(\theta(b,a), \theta(b,c))$ are defined and represent words in~$\mathcal{S}^*$ that are equivalent under the set of relations~$\mathcal{R}$ (resp. that are equal as words), or neither of them is defined. 
\end{definition} 

\begin{definition}[Conditional lcm]
We say that a left-cancellative (respectively right-cancellative) monoid~$M$ with no nontrivial invertible element \defn{admits conditional right-lcms} (resp. \defn{admits conditional left-lcms}) if any two elements of $M$ that admit a common right-multiple (resp. left-multiple) admit a common right-lcm (resp. left-lcm). 
\end{definition}

\begin{prop}[{see~\cite[Proposition~II.4.16]{Garside}}]\label{cancellative_criterion}
If $\langle\mathcal{S}, \mathcal{R}\rangle$ is a right-complemented presentation of a monoid~$M$ with syntactic right-complement~$\theta$, and if $M$ is right-Noetherian and the $\theta$-cube condition holds for every triple of pairwise distinct elements of $\mathcal{S}$, then $M$ is left-cancellative, and admits conditional right-lcms. More precisely, $u$ and $v$ admit a common right-multiple if and only if $\theta(u,v)$ exists and, then, $u\theta(u,v)=v\theta(v,u)$ represents the right-lcm of these elements.  
\end{prop}

The classical presentation of the braid group (given in Example~\ref{ex_1}) again satisfies the assumptions of the above proposition: for more details and an explicit check of the $\theta$-cube condition, we refer the reader to~\cite[Example~II.4.20]{Garside}.

For later use we also state the following result:

\begin{lemma}[{see~\cite[Lemma~II.2.22]{Garside}}]\label{lemma_gcd}
If $M$ is cancellative and admits conditional right-lcms (respectively left-lcms), then any two elements of $M$ that admit a common left-multiple (resp. right-multiple) admit a right-gcd (resp. left-gcd).  
\end{lemma}

\subsection{Garside elements and induced lattices}\label{garside_elts}

Most of the content of this section is folkloric. We include proofs for the sake of completeness. 

\begin{lemma}\label{cond_delta_lcm}
Let $M$ be a cancellative monoid with no nontrivial invertible element (so that left- and right-divisibility relations are partial orders on $M$). Assume that $M$ has conditional (left- and right-) lcms, and that $M$ has an element~$\Delta$ satisfying the following assumptions \begin{itemize}
\item the sets of left- and right-divisors of $M$ coincide, and form a finite set,
\item the set of divisors of $\Delta$ generate $M$.
\end{itemize}
Then any two elements $x, y\in M$ admit a left-lcm and a right-lcm.
\end{lemma}
\begin{proof}
As $M$ has conditional lcms, it suffices to show that any two elements~$x,y\in M$ have a (left- or right-)common multiple. We show that $x,y$ have a common right-multiple (the proof for left-multiples is similar). Note that under our assumptions, if $z$ is any divisor of $\Delta$, then $z \Delta=\Delta z'$ for some divisor~$z'$ of $\Delta$.  

Let $x=x_1 x_2\cdots x_k$, $y=y_1 y_2\cdots y_\ell$ where the $x_i$'s and $y_i$'s are divisors of $\Delta$. Without loss of generality we can assume that $\ell \leq k$ holds. Then we claim that $\Delta^k$ is a common right-multiple of $x$ and $y$. To this end, it suffices to show that if $a=a_1 a_2 \cdots a_m$ is an element of $M$ which is a product of $m$ divisors~$a_i$ of $\Delta$, then $a$ is a left-divisor of $\Delta^m$. As $a_1$ is a left-divisor of $\Delta$ and left and right-divisors of $\Delta$ coincide, we can write $\Delta^m= a_1 \Delta^{m-1} b_1$, where $b_1$ is a divisor of $\Delta$. Iterating, we eventually end up with a decomposition~$\Delta^m= a_1 a_2 \cdots a_m b_m \cdots b_2 b_1$, which concludes the proof. 
\end{proof}

\begin{definition}[Garside element]
If $M$ and $\Delta$ satisfy the assumptions of the above lemma, we say that $\Delta$ is a \defn{Garside element} in $M$. In this case we denote by $\Div(\Delta)$ the set of left-divisors of $\Delta$ (which is equal to the set of right-divisors of $\Delta$). We call its elements the \defn{simples} of $(M,\Delta)$. 
\end{definition}

\begin{figure}
\begin{multicols}{3}
\begin{pspicture}(0,0)(3.5,4.5)
\psdots(1.5,0)(0.5,1.5)(2.5,1.5)(0.5,3)(2.5,3)(1.5,4.5)
\psline(1.5,0)(0.5,1.5)
\psline(1.5,0)(2.5,1.5)
\psline(0.5,1.5)(0.5,3)
\psline(2.5,1.5)(2.5,3)
\psline(2.5,3)(1.5,4.5)
\psline(0.5,3)(1.5,4.5)
\rput(1,0){\tiny $1$}
\rput(0.1,1.5){\tiny $\sigma_1$}
\rput(-0.05,3){\tiny $\sigma_1 \sigma_2$}
\rput(0.7,4.5){\tiny $\sigma_1 \sigma_2 \sigma_1$}
\rput(3,1.5){\tiny $\sigma_2$}
\rput(3.1,3){\tiny $\sigma_2 \sigma_1$}

\end{pspicture}

\begin{pspicture}(0,0)(3,4.5)
\psdots(1.5,0.5)(0,2)(1.5,2)(3,2)(1.5,3.5)
\psline(1.5,0.5)(1.5,3.5)
\psline(1.5,0.5)(0,2)
\psline(0,2)(1.5,3.5)
\psline(1.5,0.5)(3,2)
\psline(3,2)(1.5,3.5)
\rput(1,0.5){\tiny $1$}
\rput(-0.75,2){\tiny $\sigma_1\sigma_2\sigma_1^{-1}$}
\rput(1.1,2){\tiny $\sigma_1$}
\rput(2.6,2){\tiny $\sigma_2$}
\rput(0.9,3.5){\tiny $\sigma_1 \sigma_2$}

\end{pspicture}

\begin{pspicture}(0.4,0)(3.5,4.5)
\psdots(1.5,0)(0.5,0.7)(2.5,1.4)(0.5,2.1)(2.5,2.1)(0.5,2.8)(2.5,3.5)(1.5,4.2)
\psline(1.5,0)(0.5,0.7)
\psline(0.5,0.7)(0.5,2.8)
\psline(1.5,0)(2.5,1.4)
\psline(2.5,1.4)(0.5,2.8)
\psline(2.5,1.4)(2.5,3.5)
\psline(2.5,3.5)(1.5,4.2)
\psline(0.5,2.8)(1.5,4.2)
\rput(1,0){\tiny $1$}
\rput(0.15,0.7){\tiny $\sigma_1$}
\rput(-0.2,2.1){\tiny $\sigma_1 \sigma_1 \sigma_2$}
\rput(-0.2,2.8){\tiny $(\sigma_1 \sigma_2)^2$}
\rput(0.55,4.2){\tiny $(\sigma_1 \sigma_2 \sigma_1)^2$}
\rput(3,1.4){\tiny $\sigma_1 \sigma_2$}
\rput(3.1,2.1){\tiny $\sigma_1 \sigma_2 \sigma_1$}
\rput(3.4,3.5){\tiny $\sigma_1 \sigma_2 \sigma_1 \sigma_1 \sigma_2$}

\end{pspicture}

\end{multicols}
\caption{The lattice of simples (for left-divisibility) in three different Garside monoids for $\mathcal{B}_3$, expressed in terms of the classical Artin generators of $\mathcal{B}_3$. The lattice for the classical Garside structure is on the left, the one for the dual Garside structure in the middle, and the one for the exotic Garside structure given by the monoid $\Sigma_2$ discussed in Example~\ref{ex_2} on the right.}
\label{b3}
\end{figure}
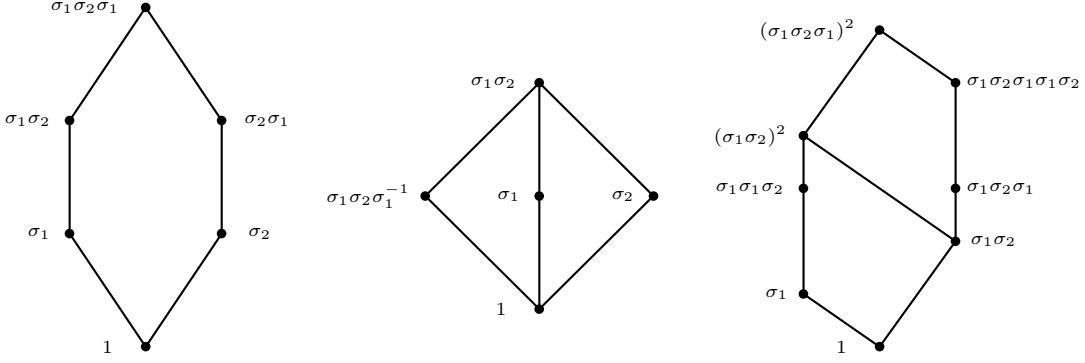

Note that if the conditions in Lemma~\ref{cond_delta_lcm} are satisfied, then the set of divisors of the Garside element~$\Delta$, endowed with the restriction of the left-divisibility relation (which is a partial order by Lemma~\ref{lem_partial}), forms a lattice. In Figure~\ref{b3}, we represented (the Hasse diagram of) the lattice induced by left-divisiblity on the set of simples in three different Garside monoids for~$\mathcal{B}_3$ given in Example~\ref{ex_2}. 

The same holds for the restriction of right-divisibility. In general these two lattices are not isomorphic. We shall see an example of this phenomenon in Remark~\ref{not_isom} below (note that in the three examples depicted in Figure~\ref{b3}, they are isomorphic). Nevertheless, we have: 

\begin{lemma}\label{lem_dual_lattice}
Let $M$ and $\Delta$ satisfying the assumptions of Lemma~\ref{cond_delta_lcm}. Let $\leq_L$ (respectively $\leq_R$) be the partial order induced by left-divisibility on $\Div(\Delta)$ (respectively by right-divisibility). Then the map~$x\mapsto \Delta x^{-1}$ is an isomorphism of lattices~$(\Div(\Delta), \leq_L) \cong (\Div(\Delta), \leq_R)^{\mathrm{op}}$. In other words, the lattice~$(\Div(\Delta), \leq_L)$ is isomorphic to the dual of the lattice~$(\Div(\Delta), \leq_R)$.   
\end{lemma}

\begin{proof}
The fact that $x\mapsto \Delta x^{-1}$ is well-defined is clear, as left- and right-divisors of $\Delta$ coincide and $M$ is cancellative. It is invertible, with inverse given by~$y\mapsto y^{-1} \Delta$. It remains to show that both~$x\mapsto \Delta x^{-1}$ and its inverse are order-preserving. Let~$x,y\in \Div(\Delta)$ satisfying $x\leq_L y$. Then there is $a\in \Div(\Delta)$ satisfying $xa=y$. As $a\in\Div(\Delta)$, there is $b\in\Div(\Delta)$ satisfying $ba=\Delta$. Similarly, as $b\in\Div(\Delta)$, there is $c\in\Div(\Delta)$ satisfying $cb=\Delta$. We then have 
$c \Delta y^{-1}= c \Delta a^{-1} x^{-1}=cb x^{-1} = \Delta x^{-1}$, which shows that $\Delta y^{-1} \leq_R \Delta x^{-1}$, hence that $x \mapsto \Delta x^{-1}$ is order-preserving. 

The proof that the inverse map $y\mapsto y^{-1} \Delta$ is also order-preserving is similar.   
\end{proof}

\begin{prop}\label{cor_dual}
Let $\Delta$ be a Garside element in $M$. Then $$(\Div(\Delta), \leq_L)\text{~is self-dual~}\Leftrightarrow (\Div(\Delta), \leq_L)\cong (\Div(\Delta), \leq_R)\Leftrightarrow(\Div(\Delta), \leq_R)\text{~is self-dual}.$$   
\end{prop}

\section{Garside structures on torus knot groups}\label{garside_torus}

Let $n,m$ be two relatively prime positive integers. The~\defn{torus knot group} $G(n,m)$ is defined as the knot group of the torus knot $T_{n,m}$, \emph{i.e.}, as the fundamental group of the complement of the torus knot $T_{n,m}$ (see~\cite[Chapter~3]{Rolfsen}). As $T_{n,m}\cong T_{m,n}$, one has $G(n,m)\cong G(m,n)$. The most basic presentation of $G(n,m)$ is given by two generators $x,y$ and a single relation $x^n=y^m$. By~\cite[Example~4]{DP}, the monoid with the same presentation is known to be a Garside monoid with $\Delta=x^n=y^m$. In particular, the group $G(n,m)$ is a Garside group. Its center is known to be infinite cyclic, generated by $x^n=y^m$.  

Another Garside presentation of $G(n,m)$ is given by \begin{equation}\label{class_knot}\langle x_1, x_2, \dots, x_n \ \vert\ 
\underbrace{x_1 x_2 \cdots}_{m~\text{factors}} = \underbrace{x_2 x_3\cdots}_{m~\text{factors}} = \dots = \underbrace{x_n x_1 \cdots}_{m~\text{factors}}
\ \rangle,\end{equation}

where in the relations the indices are taken modulo $n$ if $n<m$. The monoid with the same presentation is indeed a Garside monoid by~\cite[Example~5]{DP}. Note that, as observed in~\cite[Section~6.4, Problem~10]{Dual}, this yields in fact two distinct Garside structures on $G(n,m)$, since $G(n,m)\cong G(m,n)$. As suggested in~\textit{loc. cit.}, for $n<m$, we may call the monoid defined by the presentation~\eqref{class_knot} the~\defn{classical} Garside monoid for~$G(n,m)$, and the monoid with the same presentation but with the roles of $n$ and $m$ reversed the~\defn{dual} Garside monoid for~$G(n,m)$. Indeed, in the cases where $G(n,m)$ is an Artin--Tits group, that is, for~$n=2$ and $m$ odd where it is isomorphic to the Artin--Tits group of type~$I_2(m)$, one recovers the classical and dual braid monoids. 

An alternative Garside structure for~$G(n,m)$, which is similar to the one given by~\eqref{class_knot} but distinct in general, can be found in~\cite[Proposition~4.1]{Picantin_torus}. Additional Garside structures for some specific torus knot groups can also be found in Section~5 of \emph{loc. cit.}.

For $m=n+1$, we will construct a new Garside structure on $G(n,m)$ in the next section. We will explain how the various presentations above are related in Section~\ref{link_pres} below.    

\section{A new Garside structure on $(n,n+1)$-torus knot groups}\label{sec:main}

We now define our main object of study. 

\subsection{Definition and several presentations of the monoid}\label{sec:pres}

\begin{definition}
Let $n\geq 1$. Consider the monoid $M_n$ defined by the presentation \begin{equation}\label{eq_main}\langle~ \rho_1, \rho_2, \dots, \rho_{n} \ \vert\ \rho_1 \rho_{n} \rho_i= \rho_{i+1} \rho_{n}~\text{for~}1\leq i <n~\rangle.\end{equation} We will denote the set of generators by~$\S$, and the above set of relations by $\R$, omitting the dependency on $n$. 
\end{definition}

The group with the same presentation is in fact isomorphic to the torus knot group $G(n,n+1)$; it is indeed straightforward to check the following: 

\begin{prop}\label{prop_isom_torus}
The map $\rho_i\mapsto x^iy^{-i}\text{~for~}1\leq i \leq n$ extends to a group isomorphism between the group with presentation~\eqref{eq_main} and the $(n,n+1)$-torus knot group $G(n,n+1)=\langle~ x,y \ \vert x^n=y^{n+1} ~\rangle$.
\end{prop}

Note that $M_2=\Sigma_2$, the exotic Garside monoid for~$\mathcal{B}_3$ given in Example~\ref{ex_2}. 


It is well-known that the group~$G(n,n+1)$ is an extension of $\mathcal{B}_{n+1}$. In terms of the above presentation the map is defined as follows:  

\begin{prop}\label{prop:surjective}
The assignment $\rho_i\mapsto \sigma_1 \sigma_2 \cdots \sigma_i,$ for $i=1, \dots, n$, extends to a surjective group homomorphism $\varphi_n: G(n,n+1) \longrightarrow \mathcal{B}_{n+1}$. 
\end{prop}


\begin{proof}
It suffices to show that the elements $S_i:=\sigma_1 \sigma_2 \cdots \sigma_i\in \mathcal{B}_{n+1}$ ($i=1, \dots, n$) satisfy the defining relations of~\eqref{eq_main}. We show that $S_1 S_n S_i=S_{i+1} S_n$ for all $i=1, \dots, n-1$ by induction on  $1\leq i \leq n-1$. We have
\begin{align*}
S_1 S_{n} S_1&=\sigma_1 (\sigma_1 \sigma_2 \cdots \sigma_{n}) \sigma_1=\sigma_1 \sigma_1 \sigma_2 \sigma_1 (\sigma_3 \sigma_4 \cdots \sigma_{n})\\&=\sigma_1 \sigma_2 (\sigma_1 \sigma_2 \sigma_3\cdots \sigma_{n})=S_2 S_{n},
\end{align*}
hence the result holds for $i=1$. Now let $1< i\leq n-1$. By induction we have
\begin{align*}
S_1 S_{n} S_i&=S_1 S_{n} S_{i-1} \sigma_i=S_i S_{n} \sigma_i= S_i (\sigma_1 \sigma_2 \cdots \sigma_{n}) \sigma_i\\
&=S_i (\sigma_1 \sigma_2 \cdots \sigma_i\sigma_{i+1}) \sigma_i (\sigma_{i+2} \cdots \sigma_{n})\\&=S_i (\sigma_1 \sigma_2 \cdots \sigma_{i-1}) \sigma_{i+1} (\sigma_i \sigma_{i+1} \cdots \sigma_{n})= S_i \sigma_{i+1} S_{n}= S_{i+1} S_{n}, 
\end{align*}
which concludes the proof. 
\end{proof}

\begin{lemma}\label{noeth}
The map $\lambda :\{\rho_1, \rho_2, \dots, \rho_{n}\}\longrightarrow \mathbb{Z}_{\geq 0}$, $\rho_i\mapsto i$ extends to a uniquely defined length function $\lambda$ on $M_n$ satisfying $\lambda(ab)=\lambda(a)+\lambda(b)$ for all $a,b\in M_n$. In particular, the divisibility in $M_n$ is Noetherian, and $M_n$ is both left- and right-Noetherian.  
\end{lemma}

\begin{proof}
It suffices to show that the extension of $\lambda$ to $\S^*$ takes the same value on each side of any given relation in $\R$, in other words, that the relations in $\R$ are homogeneous with respect to $\lambda$. This is clear, as $\lambda(\rho_1 \rho_n \rho_i)=n+i+1=\lambda(\rho_{i+1} \rho_n)$ for all $1\leq i \leq n-1$.  
\end{proof}

Unfortunately, the cancellativity criteria that we recalled in Subsection~\ref{cancel} do \textit{not} work with the presentation $\langle\S, \R\rangle$ of $M_n$. We need to enlarge the set $\R$ of relations, thereby making it redundant, to be able to apply such criteria. We will need two distinct enlarged sets of relations, one to show left-cancellativity, the other one to show right-cancellativity. We introduce them in the following two Lemmata.  

\begin{lemma}\label{rel_bis}
Let $1\leq i<j \leq n$. In $M_n$ (and hence in $G(n,n+1)$), we have $$\rho_i \rho_{n}^i \rho_{j-i}=\rho_j \rho_{n}^i.$$
\end{lemma}

\begin{proof}
Note that when $i=1$, the relations are just the defining relations of $M_n$. Assume $i>1$. As $\rho_k \rho_{n}=\rho_1 \rho_{n} \rho_{k-1}$ holds for $2\leq k \leq n$, we have \begin{eqnarray}\label{eq:in}\rho_i \rho_{n}^i= \rho_1 \rho_{n} \rho_{i-1} \rho_{n}^{i-1}=\rho_1 \rho_{n} \rho_1 \rho_{n} \rho_{i-2} \rho_{n}^{i-2}= \ldots = (\rho_1 \rho_{n})^i.\end{eqnarray} Hence we get \begin{eqnarray}\label{eq_1}\rho_i \rho_{n}^i \rho_{j-i}=(\rho_1 \rho_{n})^i \rho_{j-i}.\end{eqnarray} Similarly, as $j>i$, we have \begin{eqnarray}\label{eq_2}\rho_j \rho_{n}^i= \rho_1 \rho_{n} \rho_{j-1} \rho_{n}^{i-1}=\rho_1 \rho_{n} \rho_1 \rho_{n} \rho_{j-2} \rho_{n}^{i-2}=\ldots =(\rho_1 \rho_{n})^i \rho_{j-i}.\end{eqnarray} Putting~\eqref{eq_1} and \eqref{eq_2} together we get $\rho_i \rho_{n}^i \rho_{j-i}=(\rho_1 \rho_{n})^i \rho_{j-i}=\rho_j \rho_{n}^i$. This concludes the proof.      
\end{proof}

\begin{lemma}\label{rel_bis_2}
Let $1\leq i<j \leq n$. In $M_n$ (and hence in $G(n,n+1)$), we have $$(\rho_1 \rho_n)^{n-j+1} \rho_i=\rho_{n-j+i+1} (\rho_1 \rho_n)^{n-j} \rho_j.$$
\end{lemma}
\begin{proof}
Note that when $j=n$, the claimed relations are just the defining relations of $M_n$. Assume $j<n$. As $\rho_1 \rho_n \rho_k=\rho_{k+1} \rho_n$ holds for $1\leq k<n$, we have \begin{equation}\label{eq_3}(\rho_1 \rho_n)^{n-j+1}\rho_i=(\rho_1 \rho_n)^{n-j} \rho_{i+1} \rho_n = (\rho_1 \rho_n)^{n-j-1} \rho_{i+2} \rho_n^2 =\ldots = \rho_{n-j+i+1}\rho_n^{n-j+1}.\end{equation}
Applying the same relation, we also get \begin{equation}\label{eq_4}(\rho_1 \rho_n)^{n-j} \rho_j= (\rho_1 \rho_n)^{n-j-1} \rho_{j+1} \rho_n=\ldots=\rho_n^{n-j+1}.\end{equation}
Putting~\eqref{eq_3} and \eqref{eq_4} together we get $$\rho_{n-j+i+1}(\rho_1 \rho_n)^{n-j} \rho_j=\rho_{n-j+i+1}\rho_n^{n-j+1} =(\rho_1 \rho_n)^{n-j+1}\rho_i,$$ which concludes the proof.   
\end{proof}

\begin{prop}
The monoid $M_n$ has two presentations $\langle \S, \R'\rangle$ and $\langle\S, \R''\rangle$, where $\S$ is as before the set $\{\rho_1, \rho_2, \dots, \rho_n\}$ and $\R'$ (respectively $\R''$) is the set of relations given in the statement of Lemma~\ref{rel_bis} (respectively Lemma~\ref{rel_bis_2}). 
\end{prop}

\begin{proof}
We have seen in Lemmata~\ref{rel_bis}, \ref{rel_bis_2} that all the relations in $\R'$, $\R''$ follow from the relations in $\R$, and as they contain all the relations in $\mathcal{R}$ and $M_n=\langle\S, \R\rangle$, the claim is immediate. 
\end{proof}

\subsection{Cancellativity}

\subsubsection{Left-cancellativity} In order to show that the monoid $M_n$ is left-cancellative, we will apply Proposition~\ref{cancellative_criterion} using the presentation~$\langle\mathcal{S}, \mathcal{R}'\rangle$ defined above, which is right-complemented. The presentation~$\langle \S, \R\rangle$ is also right-complemented, but it is easy to see that the $\theta$-cube condition fails for this presentation. 

Note that in the presentation~$\langle \S, \mathcal{R}'\rangle$, we have precisely one relation for each pair of indices~$i,j\in\{1,2,\dots, n\}$, $i<j$, namely $\rho_i \rho_n^i \rho_{j-i}=\rho_j \rho_n^i$. Hence $\theta$ is defined over all~$\mathcal{S}\times\mathcal{S}$, and for $i<j$ we have \[\theta(\rho_i, \rho_j)=\rho_n^i \rho_{j-i},~~\theta(\rho_j, \rho_i)=\rho_n^i.\]
\begin{lemma}\label{cube_right}
The presentation $\langle\mathcal{S}, \mathcal{R}'\rangle$ satisfies the sharp $\theta$-cube condition for every triple $(\rho_i, \rho_j, \rho_k)$ of pairwise distinct generators in $\mathcal{S}$. 
\end{lemma}

\begin{proof}
We need to check that either both $\theta( \theta(\rho_i,\rho_j), \theta(\rho_i,\rho_k))$ and $\theta( \theta(\rho_j,\rho_i), \theta(\rho_j,\rho_k))$ are defined and equal as words in $\mathcal{S}^*$, or neither is defined.

It is sufficient to distinguish three cases: the case~$i<j<k$, the case~$i<k<j$, and the case~$k<j<i$. The three remaining cases are indeed obtained for free by swapping the roles of $i$ and $j$.

\begin{itemize}
\item\noindent{\bf Case $i<j<k$}. We have 
\[ \theta (\theta(\rho_i, \rho_j), \theta(\rho_i, \rho_k))=\theta( \rho_n^i \rho_{j-i}, \rho_n^i \rho_{k-i})=\theta(\rho_{j-i}, \rho_{k-i})=\rho_n^{j-i} \rho_{k-j}, \]
where for the middle equality we used the fact that for all $a,b,c\in\mathcal{S}^*$, we have $\theta(ab,ac)=\theta(b,c)$ (which is an easy consequence of the relations~\eqref{c}-\eqref{d}). We also have   
\[ \theta (\theta(\rho_j, \rho_i), \theta(\rho_j, \rho_k))=\theta(\rho_n^i, \rho_n^j \rho_{k-j})=\theta(1,\rho_n^{j-i} \rho_{k-j})=\rho_n^{j-i} \rho_{k-j}.\]
\item\noindent{\bf Case $i<k<j$}. We have
\[ \theta(\theta(\rho_i, \rho_j), \theta(\rho_i, \rho_k))=\theta( \rho_n^i \rho_{j-i}, \rho_n^i \rho_{k-i})=\theta(\rho_{j-i}, \rho_{k-i})=\rho_n^{k-i}, \]
and 
\[ \theta(\theta(\rho_j, \rho_i), \theta(\rho_j, \rho_k))=\theta( \rho_n^i, \rho_n^k)=\theta(1, \rho_n^{k-i})=\rho_n^{k-i}. \]
\item \noindent{\bf Case $k<j<i$}. We have 
\[ \theta(\theta(\rho_i, \rho_j), \theta(\rho_i, \rho_k))=\theta(\rho_n^j, \rho_n^k)=\theta(\rho_{n}^{j-k},1)=1,\]
and 
\[ \theta(\theta(\rho_j, \rho_i), \theta(\rho_j, \rho_k))=\theta(\rho_n^j \rho_{i-j}, \rho_n^k)=\theta(\rho_n^{j-k} \rho_{i-j}, 1)=1.\]
\end{itemize}
Hence in all cases we have $\theta( \theta(\rho_i,\rho_j), \theta(\rho_i,\rho_k))=\theta( \theta(\rho_j,\rho_i), \theta(\rho_j,\rho_k))$, which concludes the proof. 
\end{proof}

\begin{prop}\label{left_cancellative}
The monoid $M_n$ is left-cancellative and admits conditional right-lcms. When it exists, the right-lcm of $u$ and $v\in M_n$ is given by $u\theta(u,v)=v\theta(v,u)$.
\end{prop}
\begin{proof}
Since $M_n$ is right-Noetherian (Lemma~\ref{noeth}) and the presentation $\langle\mathcal{S}, \mathcal{R}'\rangle$ satisfies the (sharp) $\theta$-cube condition for every triple of pairwise distinct generators in $\mathcal{S}$ (Lemma~\ref{cube_right}), Proposition~\ref{cancellative_criterion} ensures that $M_n$ is left-cancellative and admits conditional right-lcms. 
\end{proof}

\begin{cor}\label{cor:lcm}
Let $1\leq i < j \leq n$. The right-lcm of $\rho_i$ and $\rho_j$ is given by $$\rho_i \rho_n^i \rho_{j-i}=\rho_j \rho_n^i.$$
\end{cor}

\begin{proof}
This follows immediately from the proposition above, as $\theta(\rho_i, \rho_j)$ is defined and equal to $\rho_n^i \rho_{j-i}$.  
\end{proof}

\subsubsection{Right-cancellativity} Unlike many classical examples of Garside monoids (like the positive braid monoid, or more generally Artin--Tits monoids of spherical type), the defining presentation~$\langle \S, \R \rangle$ of the monoid~$M_n$ is \textit{not} symmetric for~$n\geq 3$. We therefore cannot deduce right-cancellativity from left-cancellativity. To show that $M_n$ is right-cancellative, we will show the equivalent statement that the opposite monoid~$M_n^\mathrm{op}$ is left-cancellative. This monoid has the same set of generators as~$M_n$ but we will denote them $\T=\{ \tau_i \}_{i=1, \dots, n}$ to distinguish them (with $\tau_i$ corresponding to $\rho_i$ for all $i$), and relations~$\R^{\mathrm{op}}$ which are obtained from $\mathcal{R}$ by reversing all the words. 

Recall the presentations~$\langle \S, \R\rangle$, $\langle \S, \R'\rangle$ and $\langle \S, \R''\rangle$ of $M_n$ (see Section~\ref{sec:pres}). As for left-cancellativity, it is not hard to see that the $\theta$-cube condition fails with the right-complemented presentation~$\langle \T, \R^{\mathrm{op}}\rangle$ of~$M_n^{\mathrm{op}}$, hence one cannot apply Proposition~\ref{cancellative_criterion} with this choice of presentation. Moreover, the presentation~$\langle \T, {(\R')}^{\mathrm{op}}\rangle$ which is the opposite of the presentation~$\langle \S, \R'\rangle$ that we used to show left-cancellativity is not right-complemented as in general there is more than one relation of the form~$\tau_i \cdots = \tau_j \cdots$ for~$i\neq j$ (for instance $\tau_1 \tau_3 \tau_1= \tau_3 \tau_2\text{~~and~~}\tau_3^3=\tau_1 \tau_3^2 \tau_2$ for~$n=3$), hence again Proposition~\ref{cancellative_criterion} cannot be applied with this choice of presentation. But the presentation~$\langle \T, {(\R'')}^{\mathrm{op}}\rangle$ of $M_n^{\mathrm{op}}$ is right-complemented. The set of relations~${(\R'')}^{\mathrm{op}}$ is indeed given by $$\tau_i (\tau_n \tau_1)^{n-j+1}=\tau_j (\tau_n \tau_1)^{n-j} \tau_{n-j+i+1},~\text{for~} 1\leq i < j \leq n.$$

The syntactic right-complement~$\eta$ attached to the right-complemented presentation~$\langle \T, {(\R'')}^{\mathrm{op}}\rangle$ is then given by $$\eta(\tau_i, \tau_j)=(\tau_n \tau_1)^{n-j+1}\text{ and}~\eta(\tau_j, \tau_i)=(\tau_n \tau_1)^{n-j} \tau_{n-j+i+1}~\text{for~} 1\leq i < j \leq n.$$

\begin{lemma}\label{cube_right_bis}
The presentation~$\langle\mathcal{T}, {(\R'')}^{\mathrm{op}}\rangle$ satisfies the sharp $\eta$-cube condition for every triple~$(\tau_i, \tau_j, \tau_k)$ of pairwise generators in $\mathcal{T}$. 
\end{lemma}

\begin{proof}
We proceed as in Lemma~\ref{cube_right}. It is sufficient to distinguish three cases: the case~$i<j<k$, the case~$i<k<j$, and the case~$k<j<i$. 

\begin{itemize}
\item {\bf Case $i<j<k$}. We have 
\[ \eta (\eta(\tau_i, \tau_j), \eta(\tau_i, \tau_k))=\eta( (\tau_n \tau_1)^{n-j+1}, (\tau_n \tau_1)^{n-k+1})=\eta( (\tau_n \tau_1)^{k-j},1)=1,\]
and 
\begin{align*} \eta (\eta(\tau_j, \tau_i), \eta(\tau_j, \tau_k))&=\eta( (\tau_n \tau_1)^{n-j} \tau_{n-j+i+1}, (\tau_n \tau_1)^{n-k+1} )\\&=\eta( (\tau_n \tau_1)^{k-j-1} \tau_{n-j+i+1}, 1)=1.
\end{align*}

\item {\bf Case $i<k<j$}. We have\[\eta(\eta(\tau_i, \tau_j), \eta(\tau_i, \tau_k))=\eta((\tau_n \tau_1)^{n-j+1},(\tau_n \tau_1)^{n-k+1})=\eta(1, (\tau_n\tau_1)^{j-k})=(\tau_n\tau_1)^{j-k},\]
and 
\begin{align*}
\eta(\theta(\tau_j, \tau_i), \eta(\tau_j, \tau_k))&=\eta((\tau_n \tau_1)^{n-j} \tau_{n-j+i+1}, (\tau_n \tau_1)^{n-j} \tau_{n-j+k+1})\\&=\eta(\tau_{n-j+i+1}, \tau_{n-j+k+1})=(\tau_n \tau_1)^{j-k}.
\end{align*}
\item {\bf Case $k<j<i$}. We have 
\begin{align*}
\eta(\eta(\tau_i, \tau_j), \eta(\tau_i, \tau_k))&=\eta((\tau_n \tau_1)^{n-i} \tau_{n-i+j+1}, (\tau_n \tau_1)^{n-i} \tau_{n-i+k+1})\\
&=\eta(\tau_{n-i+j+1},\tau_{n-i+k+1})= (\tau_n \tau_1)^{i-j-1} \tau_{n-j+k+1},
\end{align*}
and
\begin{align*} \eta(\eta(\tau_j, \tau_i), \eta(\tau_j, \tau_k))&=\eta((\tau_n \tau_1)^{n-i+1}, (\tau_n \tau_1)^{n-j} \tau_{n-j+k+1})\\&=\eta(1, (\tau_n \tau_1)^{i-j-1} \tau_{n-j+k+1})=(\tau_n \tau_1)^{i-j-1} \tau_{n-j+k+1}.
\end{align*}
\end{itemize}\end{proof}

\begin{prop}\label{right_cancellative}
The monoid $M_n^{\mathrm{op}}$ is left-cancellative and admits conditional right-lcms. Equivalently, the monoid $M_n$ is right-cancellative and admits conditional left-lcms.  
\end{prop}
\begin{proof}
Since $M_n^{\mathrm{op}}$ is right-Noetherian (as $M_n$ is left-Noetherian by Lemma~\ref{noeth}) and the presentation $\langle\mathcal{T}, {(\R'')}^{\mathrm{op}}\rangle$ satisfies the (sharp) $\theta$-cube condition for every triple of pairwise distinct generators in $\mathcal{T}$ (Lemma~\ref{cube_right_bis}), Proposition~\ref{cancellative_criterion} ensures that $M_n^{\mathrm{op}}$ is left-cancellative and admits conditional right-lcms. 
\end{proof}

We also note:

\begin{cor}\label{cor:lcm_left}
Let $1\leq i < j \leq n$. The left-lcm of $\rho_i$ and $\rho_j$ is given by $$(\rho_1\rho_n)^{n-j+1}\rho_i=\rho_{n-j+i+1}(\rho_1\rho_n)^{n-j}\rho_j.$$
\end{cor}

\begin{proof}
By Proposition~\ref{cancellative_criterion}, the right-lcm of $u$ and $v\in M_n^{\mathrm{op}}$ exists if and only if $\theta(u,v)$ is defined, and is then given by $u\theta(u,v)=v\theta(v,u)$. For $\tau_i$ and $\tau_j$ in $M_n^{\mathrm{op}}$ we know that $\theta(\tau_i,\tau_j)$ is defined, hence the right-lcm of $\tau_i$ and $\tau_j$ is given by $\tau_i \theta(\tau_i, \tau_j)=\tau_j \theta(\tau_j, \tau_i)$. It then suffices to reverse the obtained words to get the left-lcm of $\rho_i$ and $\rho_j$ in $M_n$.  
\end{proof}

\subsection{Garside structure}

In this section, we establish the existence of a Garside element in $M_n$, and deduce from it and from previously shown properties that $(M_n,\rho_n^{n+1})$ is a Garside monoid.  

\begin{nota}
Let $M_n$ be the monoid with the presentation~$\langle\mathcal{S}, \mathcal{R}\rangle$ as defined in Section~\ref{sec:pres}. We set $\Delta:=\rho_n^{n+1}$, omitting the dependency on $n$.  
\end{nota}

\begin{prop}\label{garside_divisors}
The following holds in $M_n$:
\begin{enumerate}
\item We have $\rho_1 (\rho_n \rho_1)^{n-1}=\rho_n^n$. Hence $\Delta=(\rho_1 \rho_n)^n=(\rho_n \rho_1)^n.$
\item Let $1\leq i \leq n$. Set $a_i:=\rho_n^i (\rho_1 \rho_n)^{n-i}$. Then $\rho_i a_i = a_i \rho_i=\Delta$. In particular, every element in $\mathcal{S}$ is both a left- and a right-divisor of $\Delta$ (and the left- and right-complements coincide), and $\Delta$ is central in $M_n$. 
\item Let $a,b\in M_n$ such that $ab=\Delta$. Then $ba=\Delta$.   
\end{enumerate}
\end{prop}

\begin{proof}
The first claim follows from the fact that for all $1\leq k \leq n-1$, we have \begin{equation}\label{eq_move}
\rho_1 (\rho_n \rho_1)^k=\rho_{k+1} \rho_n^k.\end{equation} Indeed, for $k=1$ this is just a relation in $\mathcal{R}$, while the general case is obtained by induction on $k$: $\rho_1 (\rho_n \rho_1)^k=\rho_1 \rho_n \rho_1 (\rho_n \rho_1)^{k-1}=\rho_1 \rho_n \rho_k \rho_n^{k-1}=\rho_{k+1} \rho_n^k.$

For the second claim, using the first claim and~\eqref{eq_move} me have $$\Delta=(\rho_1\rho_n)^{n}=\rho_1 (\rho_n \rho_1)^{i-1} \rho_n(\rho_1 \rho_n)^{n-i}=\rho_i \rho_n^i (\rho_1 \rho_n)^{n-i}=\rho_i a_i.$$
Arguing as for~\eqref{eq_move}, for all $k\leq n-i$, we see that $(\rho_1 \rho_n)^{k} \rho_i=\rho_{i+k} \rho_n^k$. Applying this with $k=n-i$ we get $$\Delta=\rho_n^{n+1}=\rho_n^i \rho_n \rho_n^{n-i}=\rho_n^i (\rho_1 \rho_n)^{n-i} \rho_i=a_i \rho_i,$$
which shows the second claim. 

The last claim is an immediate consequence of the cancellativity of $M_n$ and the second claim, as the property holds for the set $\mathcal{S}$ which generates $M_n$.   
\end{proof}

\begin{cor}\label{garside_finite}
The left and right-divisors of $\Delta$ coincide, and form a finite set. 
\end{cor}

\begin{proof}
The fact that the left and right-divisors of $\Delta$ coincide follows immediately from Point~$(3)$ of Proposition~\ref{garside_divisors}. The fact that this set is finite is clear by Lemma~\ref{noeth}, since $\S$ is finite. 
\end{proof}

\begin{cor}\label{lcm_s}
Both the left and the right-lcm of the generators $\rho_1, \rho_2, \dots, \rho_n$ of $M_n$ are given by $\rho_n^n=\rho_1 (\rho_n \rho_1)^{n-1}$. 
\end{cor}
In particular, using also Theorem~\ref{garside_gn} below, the family $(M_n)_{n\geq 2}$ yields an example of a family of Garside monoids where $\Delta$ is \textit{not} the lcm of the atoms.
\begin{proof}
By Corollary~\ref{cor:lcm}, we have that the right-lcm of $\rho_n$ and $\rho_{n-1}$ is given by $\rho_n^n$. Hence to conclude it suffices to show that $\rho_i$ left-divides $\rho_n^n$, for all $1\leq i \leq n-2$. This is the case, as by Point~$(1)$ of the above proposition together with Relation~\eqref{eq:in}, we have $$\rho_n=(\rho_1 \rho_n)^i (\rho_1 \rho_n)^{n-1-i} \rho_1=\rho_i \rho_n^i (\rho_1 \rho_n)^{n-1-i}\rho_1.$$

The proof that $\rho_n^n$ is also the left-lcm of the elements in $\S$ is similar. This time, consider the left-lcm of $\rho_1$ and $\rho_2$. By Corollary~\ref{cor:lcm_left}, it is equal to $(\rho_1 \rho_n)^{n-1} \rho_1$ which, by the first point of Proposition~\ref{garside_divisors}, is equal to $\rho_n^n$. Hence to conclude the proof, it suffices to check that for all $2 < j \leq n-1$, the generator $\rho_j$ is a right-divisor of $\rho_n^n$. But using Relation~\ref{eq_4} (which is valid for all $j\geq 2$), we have $$\rho_n^n=\rho_n^{j-1} \rho_n^{n-j+1}=\rho_n^{j-1}(\rho_1 \rho_n)^{n-j}\rho_j,$$ hence $\rho_j$ right-divides $\rho_n^n$.  
\end{proof}

\begin{theorem}\label{garside_gn}
The pair $(M_n, \Delta)$ is a Garside monoid. The correponding Garside group~$G(M_n)$ is isomorphic to the $(n,n+1)$-torus knot group~$G(n,n+1)$.   
\end{theorem}


\begin{proof}
The monoid $M_n$ is cancellative and admits conditional lcm's by Propositions~\ref{left_cancellative} and~\ref{right_cancellative}. It has Noetherian divisibility by Lemma~\ref{noeth}. Now by Proposition~\ref{garside_divisors} and Corollary~\ref{garside_finite}, the element $\Delta$ satisfies the last two conditions of Definition~\ref{def_garside}. We then get the existence of lcm's from the existence of conditional lcm's, applying Lemma~\ref{cond_delta_lcm}.

By Theorem~\ref{thm:ore}, we get that $G(M_n)$ has the same presentation as $M_n$, hence by Proposition~\ref{prop_isom_torus} we have $G(M_n)\cong G(n,n+1)$. \end{proof} 

\begin{rmq}\label{not_isom}
The lattice of simples of $M_3$ (for left-divisibility) is given in Figure~\ref{simples_3}. Recall that the lattice of simples of $M_2$ was given in Figure~\ref{fig_intro}. Note that the lattice in Figure~\ref{simples_3} is not self-dual; in particular, by Proposition~\ref{cor_dual} the lattice of simples of $\Delta$ for left-divisibility is not isomorphic to the lattice of simples for right-divisibility.   
\end{rmq}

\begin{figure}

\begin{center}
\begin{pspicture}(0,0)(12,13)
\rput(4,0){\tiny $1$}
\rput(4,1){\tiny $\rho_1$}
\rput(2,2){\tiny $\rho_2$}
\rput(0,3){\tiny $\rho_2 \rho_1$}
\rput(8,3){\tiny $\rho_3$}
\rput(4,4){\tiny $\rho_1 \rho_3$}
\rput(10,4){\tiny $\rho_3 \rho_1$}
\rput(2,5){\tiny $\rho_1 \rho_3 \rho_1$}
\rput(4.3,5.2){\tiny $\rho_3 \rho_2$}
\rput(0,6){\tiny $\rho_2 \rho_1 \rho_3$}
\rput(5,6.5){\tiny $\rho_3 \rho_2 \rho_1$}
\rput(7.7,6){\tiny $\rho_3^2$}
\rput(10,7){\tiny $\rho_3 \rho_1 \rho_3$}
\rput(12,7){\tiny $\rho_3^2 \rho_1$}
\rput(2,8){\tiny $(\rho_1 \rho_3)^2$}
\rput(9.8,8){\tiny $(\rho_3 \rho_1)^2$}
\rput(4,9){\tiny $\rho_3^3$}
\rput(8,9){\tiny $\rho_3 \rho_2 \rho_1 \rho_3$}
\rput(12,10){\tiny $\rho_3^2 \rho_1 \rho_3$}
\rput(8,11){\tiny $(\rho_3 \rho_1)^2 \rho_3$}
\rput(6,12){\tiny $\rho_3^4$}
\psline(4,0.25)(4,0.75)
\psline(4,1.25)(4,3.75)
\psline(3.8,0.25)(2.25,2)
\psline(4.2,0.25)(7.8,2.9)
\psline(8.2,3.2)(9.6,3.8)
\psline(2,2.25)(2,4.75)
\psline(0,3.25)(0,5.75)
\psline(1.8,2.1)(0.4,3)
\psline(4.4,4.2)(7.4,5.8)
\psline(3.6,4.2)(2.55,5)
\psline(2,5.25)(2,7.75)
\psline(7.92,3.25)(4.3,5)
\psline(10,4.25)(10,6.75)
\psline(10,7.25)(10,7.75)
\psline(7.9,6)(11.6,7)
\psline(12,7.25)(12,9.75)
\psline(9.4,7)(4.3,9)
\psline(0,6.2)(1.7,7.75)
\psline(2.6,8.2)(3.8,8.9)
\psline(4.2,9.1)(5.8,11.9)
\psline(5,6.65)(8,8.75)
\psline(8,9.25)(8,10.75)
\psline(10,8.25)(8.2,10.75)
\psline(7.5,6.1)(4,8.75)
\psline(11.6,10.3)(6.3,12)
\psline(7.85,11.18)(6.3,11.9)
\psline(8.1,3.25)(7.7,5.75)
\psline(4.5,5.4)(9.5,7.75)
\psline(4.3,5.45)(4.8,6.3)

\end{pspicture}
\end{center} 
\caption{The lattice of divisors of the Garside element~$\Delta=\rho_3^4$ in $M_3$ for left-divisibility.}
\label{simples_3}
\end{figure}

\subsection{Link between the various presentations}\label{link_pres}

As mentioned in Section~\ref{garside_torus}, the presentations\begin{align} 
\label{cyclic} &\langle~ x_1, x_2, \dots, x_n \ \vert\  x_1 x_2 \cdots x_n x_1 = x_2 x_3 \cdots x_n x_1 x_2 = \dots = x_n x_1 x_2 \cdots x_n~\rangle\\ \label{toric} &\langle~ x,y \ \vert\ x^n=y^{n+1} ~\rangle
\end{align}
yield two distinct Garside structures on the $(n,n+1)$-torus knot group $G(n,n+1)$. We have already seen in Proposition~\ref{prop_isom_torus} how to pass from the defining presentation of $G(M_n)$ to the presentation~\eqref{toric}. For the other presentation above we have:

\begin{prop}\label{prop:iso}
The map $$\rho_1\mapsto x_1, \rho_2\mapsto x_n x_1, \rho_3\mapsto x_{n-1} x_n x_1, \dots, \rho_n \mapsto x_2 x_3 \cdots x_n x_1$$ extends to a group isomorphism $\phi$ between the group with presentation~\eqref{eq_main} and the group with presentation~\eqref{cyclic}, with inverse $\psi$ given by $$x_1 \mapsto \rho_1, x_n \mapsto \rho_2 \rho_1^{-1}, x_{n-1} \mapsto \rho_3 \rho_2^{-1}, \dots, x_2 \mapsto \rho_n \rho_{n-1}^{-1}.$$
\end{prop}

\begin{proof}
The fact that the two defined maps are inverse to each other is immediate, hence we only need to show that they extend to group homomorphisms. To this end, we first show that the $\phi(\rho_i)$'s satisfy the defining relations of presentation~\eqref{eq_main}, which is enough to conclude that $\phi$ is a homomorphism. We have $$\phi(\rho_1)\phi(\rho_n)\phi(\rho_1)=\underbrace{x_1 x_2 \cdots x_n x_1}_{=x_n x_1 x_2\cdots x_n} x_1=\phi(\rho_2)\phi(\rho_n).$$
Now let $1< i < n$. We have $$\phi(\rho_1) \phi(\rho_n) \phi(\rho_i)=x_1 x_2 x_3 \cdots x_n x_1 x_{n-i+2} \cdots x_n x_1$$
and \begin{align*}\phi(\rho_{i+1}) \phi(\rho_n)&=x_{n-i+1} \cdots x_n x_1 x_2 x_3 \cdots x_n x_1\\ &=\underbrace{(x_{n-i+1} \cdots x_n x_1 x_2 x_3 \cdots x_{n-i+1})}_{=x_1 x_2 \cdots x_n x_1} x_{n-i+2} \cdots x_n x_1=\phi(\rho_1) \phi(\rho_n) \phi(\rho_i),\end{align*}
hence $\phi$ is a homomorphism. 

Similarly, we have to show that the $\psi(x_i)$ satisfy the defining relations of Presentation~\eqref{cyclic}. We have $$\psi(x_n) \psi(x_1) \psi(x_2)\cdots \psi(x_n)=\rho_2 \rho_n \rho_1^{-1}=\rho_1 \rho_n=\psi(x_1) \psi(x_2)\cdots \psi(x_n) \psi(x_1).$$ Now let $1< i < n$. We have 
\begin{align*}
&~\psi(x_i) \psi(x_{i+1}) \cdots \psi(x_n) \psi(x_1) \cdots \psi(x_i)\\ &=(\rho_{n+2-i}\rho_{n+1-i}^{-1})(\rho_{n+1-i}\rho_{n-i}^{-1}) \cdots (\rho_2 \rho_1^{-1}) \rho_1 (\rho_n \rho_{n-1}^{-1}) (\rho_{n-1} \rho_{n-2}^{-1}) \cdots (\rho_{n+2-i}\rho_{n+1-i}^{-1})\\
&=\rho_{n+2-i}\rho_n \rho_{n+1-i}^{-1}=\rho_1 \rho_n \rho_{n+1-i}\rho_{n+1-i}^{-1}=\rho_1 \rho_n=\psi(x_n) \psi(x_1) \psi(x_2)\cdots \psi(x_n),
\end{align*}
hence $\psi$ is also a homomorphism. This concludes the proof. 
\end{proof}

\subsection{Link with braid groups of complex reflection groups}\label{sec:cplex}

The exceptional complex reflection group $G_{12}$ has three generators $s,t,u$ and relations $s^2=t^2=u^2=1$, $stus=tust=ustu$. Its braid group~$\mathcal{B}(G_{12})$ has generators $\sigma, \tau, \upsilon$ subject to the same relations as $s,t,u$ except the quadratic ones (see~\cite{BMR}). We can deduce the following from Proposition~\ref{prop:iso}:

\begin{cor}\label{cor_g12}
 The complex reflection group $G_{12}$ has a presentation with generators $r_1, r_2, r_3$ and relations $$r_1 r_3 r_1= r_2 r_3, ~r_1 r_3 r_2= r_3^2, ~r_1^2=1.$$
\end{cor}

In Corollary~\ref{prop_sym_pres} below, we give analogous presentations for the symmetric groups~$\mathfrak{S}_n$. 

\begin{proof}
It follows from Proposition~\ref{prop:iso} that $G(M_3)\cong \mathcal{B}(G_{12})$, as $\mathcal{B}(G_{12})$ has the presentation~\eqref{cyclic} for $n=3$ if we set $\sigma=x_1$, $\tau=x_2$, $\upsilon=x_3$. To show the statement, note that the given relations are exactly those of~\eqref{eq_main} (except that the generators are denoted by $r_i$ instead of $\rho_i$), with the additional relation $r_1^2=1$. By the first point $G(M_3)$ is isomorphic to the braid group of $G_{12}$. Now $G_{12}$ is obtained from $\mathcal{B}(G_{12})$ by adding the relations $\sigma^2=\tau^2=\upsilon^2=1$, but since $\sigma, \tau$ and $\upsilon$ are all conjugate in $\mathcal{B}(G_{12})$, it suffices to add the relation $\sigma^2=1$ to get a presentation of $G_{12}$; this translates into the relation $\rho_1^2=1$. 
\end{proof}

\begin{rmq}\label{rmq:cplx_gps} Corollary~\ref{cor_g12} yields a new Garside structure on $\mathcal{B}(G_{12})$. Note that the complex reflection group $G_{12}$ is not well-generated. By work of Bessis~\cite{Dualcplx}, every well-generated irreducible complex reflection group admits a dual braid monoid, in particular, the corresponding braid group is a Garside group. For $G_{12}$ and as suggested by Bessis~\cite[Section~6.4, Problem~10]{Dual} (see also Section~\ref{garside_torus} above), one can nevertheless still define a \textit{dual} braid monoid in some sense. Almost all braid groups attached to irreducible complex reflection groups which are not well-generated have been shown to be Garside groups: see Dehornoy--Paris~\cite[Proposition~5.2 and Example~5]{DP} (for $G_{15}$, $G_7$, $G_{11}$, $G_{19}$, $G(2de,2e,2)$ for $d>1$, which all have isomorphic braid group, $G_{12}$, and $G_{22}$), Picantin~\cite[Exemples 11, 13]{Picantin} (for $G_{13}$, whose braid group is isomorphic to the Artin--Tits group of type $I_2(6)=G_2$), and Corran--Lee--Lee~\cite{CLL} (for the remaining imprimitive groups). See also~\cite[Example~IX.3.25]{Garside}. It seems that the only irreducible complex reflection group for which it remains open to determine whether the corresponding braid group is a Garside group or not is $G_{31}$.
\end{rmq}

\begin{rmq}\label{rmq_complex_interval}
In view of the previous remark, it is natural to wonder if $G(n,n+1)$ is the braid group of a complex reflection group in a natural way. For $n=2$ we know that $G(2,3)\cong G(M_2)$ is isomorphic to the~$3$-strand braid group, which is the braid group of several irreducible complex reflection groups (obtained by adding the relation~$\rho_1^i=1$ for some~$i>1$ to the presentation of~$G(M_2)$). For~$i=2$ we get the symmetric group~$\mathfrak{S}_3$, and for $i=3, 4, 5$ the exceptional groups~$G_4$, $G_8$, and $G_{16}$ respectively--note that these presentations already occur in Coxeter's paper~\cite{Coxeter} from~1959. It is easy to check that the Garside monoid~$M_2$ can be obtained from the finite group~$G_4$ as an~\textit{interval group}, another method for producing Garside monoids; see~\cite[Section~0.5]{Dual} or~\cite[Chapter~VI]{Garside} for more details. Basically this method allows one to show that a monoid $M$ is a Garside monoid by realizing its lattice of simples in some (in general, but not necessarily finite) group $G$ which is a quotient of $G(M)$ (typically, for Artin-Tits groups of spherical type, both classical and dual Garside structures are obtained in this way and the group~$G$ is the corresponding Coxeter group). For~$n\geq 4$, adding the relation~$\rho_1^2=1$ to the presentation of~$G(M_n)$ seems to yield an infinite group, and the same can be expected for $i>2$. This suggests the question below.
\end{rmq}

\begin{question}
Let $n\geq 4$ and $i>1$. Consider the quotient~$\overline{G(M_n)}$ of $G(M_n)$ by the relation~$\rho_1^i=1$. Does this quotient admit a natural realization as an infinite complex reflection group ?
\end{question}

Note that the same question can be asked if we replace $\rho_1^2=1$ by $\rho_1^i=1$, $i\geq 3$ (even for $n=2$ and $n=3$ in the cases which are not covered by the above remark or Corollary~\ref{cor_g12}).

\section{Link with the braid group on $n$ strands}\label{sec:link}

In this section, we give a new presentation of the braid group~$\mathcal{B}_{n+1}$, obtained by adding suitable relations to the presentation~$\langle \S, \R \rangle$ of $G(M_n)$. Using it we show that the submonoid~$\Sigma_n$ of $\mathcal{B}_{n+1}$ generated by $\sigma_1, \sigma_1 \sigma_2, \dots, \sigma_1 \sigma_2 \cdots \sigma_n$ is an Ore monoid with group of fractions isomorphic to $\mathcal{B}_{n+1}$, and conjecture that this monoid admits an explicit finite presentation.

\begin{definition}\label{def_hn}
Let $\mathcal{H}_n^+$ be the monoid defined by the presentation \begin{equation}\label{eq:hn} \langle~\rho_1, \rho_2, \dots, \rho_n \ \vert\ \rho_1 \rho_j \rho_i=\rho_{i+1} \rho_j\text{~for~}1\leq i < j \leq n~\rangle\end{equation} 
\end{definition}

\begin{prop}\label{pres_bn}
There is an isomorphism between the group with presentation~\eqref{eq:hn} and the $(n+1)$-strand braid group~$\mathcal{B}_{n+1}$, given by $\rho_i\mapsto \sigma_1\sigma_2\cdots\sigma_i\text{~for~} 1\leq i \leq n$.
\end{prop}

\begin{proof}
We show that the assignment~$\rho_i \mapsto \sigma_1 \sigma_2 \cdots \sigma_i$, $1\leq i \leq n$, extends to a group isomorphism $f$ between the two groups. To this end, it suffices to show that $f$ extends to a group homomorphism, and that the assignment $\sigma_i \mapsto \rho_{i-1}^{-1} \rho_i$ (with the convention $\rho_0=1$) extends to a group homomorphism $g$ between $\mathcal{B}_{n+1}$ and the group with presentation~\eqref{eq:hn}. Indeed both maps are clearly inverse to each other. 

Showing that the~$f(\rho_i)$'s satisfy the claimed relations can be checked by exactly the same computation as the one given in the proof of Proposition~\ref{prop:surjective} where it is done in the case~$j=n$ (or just derived from it by invoking the embeddings~$\mathcal{B}_k \subseteq \mathcal{B}_{k+1}$). Hence $f$ is a group homomorphism. 

Conversely, let us check that the~$g(\sigma_i)$'s satisfy the braid relations. Let $1\leq i \leq n-1$. Using the relations~$\rho_1 \rho_i \rho_{i-1}=\rho_i^2$ and $\rho_1 \rho_{i+1} \rho_i= \rho_{i+1}^2$ we get $$\rho_{i-1}^{-1}\rho_i^{-1} \rho_1 \rho_{i+1} \rho_i= \rho_i^{-2} \rho_1 \rho_{i+1}^2.$$
Replacing $\rho_i^{-1} \rho_1 \rho_{i+1}$ by $\rho_{i+1} \rho_{i-1}^{-1}$ in each side (using the relation~$\rho_1 \rho_{i+1} \rho_{i-1}=\rho_i \rho_{i+1}$) we get the equality $$\rho_{i-1}^{-1} \rho_{i+1} \rho_{i-1}^{-1} \rho_i= \rho_i^{-1} \rho_{i+1} \rho_{i-1}^{-1} \rho_{i+1}.$$ The left hand side of the above equality is equal to $g(\sigma_i) g(\sigma_{i+1}) g(\sigma_i)$, while the right hand side is equal to $g(\sigma_{i+1}) g(\sigma_i) g(\sigma_{i+1})$, thus establishing the braid relation $$g(\sigma_i) g(\sigma_{i+1}) g(\sigma_i)=g(\sigma_{i+1}) g(\sigma_i) g(\sigma_{i+1}).$$ It remains to check that $g(\sigma_i) g(\sigma_j)=g(\sigma_j) g(\sigma_i)$ holds whenever $1\leq i < j-1 \leq n-1$. Using the relations~$\rho_1 \rho_j \rho_i=\rho_{i+1} \rho_j$ and $\rho_1 \rho_{j-1} \rho_{i-1}=\rho_i \rho_{j-1}$ we can write $$\rho_{i-1}^{-1} \rho_{j-1}^{-1} \rho_1^{-1} \rho_{i+1} \rho_j=\rho_{j-1}^{-1} \rho_i^{-1} \rho_1 \rho_j \rho_i.$$
Replacing $\rho_{j-1}^{-1} \rho_1^{-1} \rho_{i+1}$ by $\rho_i \rho_{j-1}^{-1}$ in the left hand side (using the relation~$\rho_1 \rho_{j-1}\rho_i=\rho_{i+1} \rho_{j-1}$) and $\rho_i^{-1} \rho_1 \rho_j$ by $\rho_j \rho_{i-1}^{-1}$ in the right hand side (using the relation~$\rho_1 \rho_j \rho_{i-1}=\rho_i \rho_j$), we get the equality $$\rho_{i-1}^{-1} \rho_i \rho_{j-1}^{-1} \rho_j=\rho_{j-1}^{-1} \rho_j \rho_{i-1}^{-1} \rho_i.$$
This equality is nothing but the equality~$g(\sigma_i) g(\sigma_j)=g(\sigma_j) g(\sigma_i)$. This shows that $g$ is a group homomorphism, and concludes the proof.     
\end{proof}

\begin{cor}\label{prop_sym_pres}
The symmetric group~$\mathfrak{S}_{n+1}$ admits the presentation \[\bigg\langle~ r_1, r_2, \dots, r_n \ \bigg\vert\ \begin{matrix} {r_1}^2=1,\\ ~r_1 r_j r_i= r_{i+1} r_j,~\text{for}~ 1\leq i < j \leq n.\end{matrix}~\bigg\rangle\] where $r_i$ corresponds to the cycle $(1,2,\dots,i+1)$ for $1\leq i \leq n$.
\end{cor}

\begin{proof}
The claimed set of relations is given by the relations in~\eqref{eq:hn}, except that we added the relation stating that the square of the first generator is equal to one. As all the~$\sigma_i$'s are conjugate in $\mathcal{B}_{n+1}$, it suffices to add to the braid relations the relation~$\sigma_1^2=1$ to get a presentation of the symmetric group~$\mathfrak{S}_{n+1}$. In view of Proposition~\ref{pres_bn} this is equivalent to adding the relation~$\rho_1^2=1$ to the set of relations given in~\eqref{eq:hn}. 
\end{proof}

Investigating the properties of the monoid $\mathcal{H}_n^+$ appears as a natural question. 

\begin{lemma}\label{lem:ore}
In the monoids $\Sigma_n$ and $\mathcal{H}_n^+$, every two elements $x,y$ admit both a common right-multiple and a common left-multiple.
\end{lemma}

\begin{proof}
This follows immediately from the fact that both $\Sigma_n$ and $\mathcal{H}_n^{+}$ are quotients of the Garside monoid~$M_n$. Indeed, the presentation of $\mathcal{H}_n^+$ is obtained from the presentation $\langle \S, \R\rangle$ of $M_n$ by adding relations, and under the isomorphism of Proposition~\ref{pres_bn}, the submonoid~$\Sigma_n$ is precisely the submonoid of the group with presentation~\eqref{eq:hn} generated by $\rho_1, \rho_2, \dots, \rho_n$, which is a quotient of $\mathcal{H}_n^+$.\end{proof}

As a corollary we get:

\begin{prop}\label{prop_ore}
The submonoid $\Sigma_n$ of $\mathcal{B}_{n+1}$ is an Ore monoid, with group of fractions isomorphic to $\mathcal{B}_{n+1}$. 
\end{prop}

\begin{proof}
For the first statement, we need cancellativity and the existence of left-multiples. The last condition is given by Lemma~\ref{lem:ore}, while cancellativity immediately follows from the fact that $\Sigma_n$ is a submonoid of a group. The second statement follows, as $\Sigma_n$ embeds into $\mathcal{B}_{n+1}$, with image generating $\mathcal{B}_{n+1}$ as a group: this ensures that the induced map $G(\Sigma_n)\longrightarrow \mathcal{B}_{n+1}$ is an isomorphism. 
\end{proof}

\begin{rmq}\label{rmq_dehornoy} It was noticed by Dehornoy~\cite[Example~3.7]{subword} that the monoid $\mathcal{H}_3^+$ does not have lcm's (and the same holds for $n>3$). Indeed, both $\rho_1 \rho_2 \rho_1=\rho_2^2$ and $\rho_1 \rho_3 \rho_1=\rho_2 \rho_3$ are common right-multiples of $\rho_1$ and $\rho_2$, and it is straightforward to check that none of these two elements left-divides the other one. Similarly, in $\Sigma_n$, both $\sigma_1 \sigma_1 \sigma_2 \sigma_1$ and $\sigma_1 \sigma_1 \sigma_2 \sigma_3 \sigma_1$ are common right-multiples of $\sigma_1$ and $\sigma_1\sigma_2$, and it is clear that none of them left-divides the other one in $\Sigma_n$. This implies that neither $\Sigma_n$ nor $\mathcal{H}_n^+$ are Garside monoids. The answer to the second part of Question~\ref{q1} from the Introduction is therefore negative.   
\end{rmq}

Dehornoy also asked whether $\mathcal{H}_3^+$ is (right-)cancellative or not (see~\cite[Question~3.8]{subword}: note that Dehornoy works with the opposite monoids of $\mathcal{H}_3^+$ and $\Sigma_3$) and conjectured that this is the case. More precisely he conjectured that $\mathcal{H}_3^+\cong \Sigma_3$. We conjecture the following more general statement, which would also imply that $\Sigma_n$ admits a finite presentation (answering the first part of Question~\ref{q1}).

\begin{conjecture}\label{conj}
Let $n\geq 3$. Then 
\begin{enumerate}
\item The monoid~$\mathcal{H}_n^+$ is cancellative,
\item The monoid~$\mathcal{H}_n^+$ is isomorphic to $\Sigma_n$ via $\rho_i \mapsto \sigma_1 \sigma_2\cdots \sigma_i$. In particular, it embeds into $\mathcal{B}_{n+1}$, which is therefore isomorphic to its group of fractions. 
\end{enumerate} 
\end{conjecture} 

\begin{rmq}
Both items of the above conjecture are actually equivalent: clearly $(2)\Rightarrow(1)$ as $\Sigma_n$ is cancellative. Conversely, assume that $\mathcal{H}_{n}^+$ is cancellative. Then by Lemma~\ref{lem:ore} it is an Ore monoid, embedding into its group of fractions~$G(\mathcal{H}_n^+)$, and by Ore's Theorem~\ref{thm:ore}, we get that $G(\mathcal{H}_n^+)$ is isomorphic to the group with presentation~\ref{eq:hn}, which by Proposition~\ref{pres_bn} is isomorphic to $\mathcal{B}_{n+1}.$ The submonoid $\mathcal{H}_n^+$ of $\mathcal{B}_{n+1}$ then precisely corresponds under this isomorphism to the submonoid of $\mathcal{B}_{n+1}$ generated by $\sigma_1, \sigma_1\sigma_2, \dots, \sigma_1\sigma_2\cdots\sigma_n$, \emph{i.e.}, to $\Sigma_n$. 
\end{rmq}
 

\section{Related Garside structures on dihedral Artin--Tits groups of odd type}\label{sec:dihedral}

While the exotic Garside structure on $\mathcal{B}_3$ given by $\Sigma_2$ (see Example~\ref{ex_2}), which was generalized in the previous sections to the groups~$(G(M_n))_{n\geq 1}$, does not seem to generalize to Artin--Tits groups of type~$A_n$ for $n\geq 2$ (see the previous section), it is natural to wonder which Artin--Tits groups of spherical type (or more generally braid groups of complex reflection groups) admit a Garside structure analogous to the one introduced for $G(M_n)$. 

The case of dihedral Artin--Tits groups appears to us as the first family to consider, as they are the Artin--Tits groups with the most elementary structure, and $\mathcal{B}_3$ is an Artin--Tits group of dihedral type. The aim of this section is to show that dihedral Artin--Tits groups of odd type admit a Garside structure similar to the one obtained for $G(M_n)$. These Garside structures are presumably new. 

Let $m\geq 3$ be odd. Recall that the dihedral group $I_2(m)$ is generated by two simple reflections $s, t$ subject to the relations $s^2=1=t^2$ and the braid relation $\underbrace{st \cdots}_{m~\text{factors}} = \underbrace{ts \cdots}_{m~\text{factors}}$. The corresponding Artin--Tits group~$\mathcal{B}(I_2(m))$ is generated by $\sigma, \tau$, only subject to the braid relation of $I_2(m)$. Note that $\mathcal{B}(I_2(m))\cong G(2,m)$, the $(2,m)$-torus knot group. 

For $m$ an integer as above, we denote by $M(m)$ the monoid generated by two elements $\rho_1, \rho_2$, and subject to the relation~$\rho_1 \rho_2^{(m-1)/2} \rho_1 = \rho_2^{(m+1)/2}$. We denote by $B(m)$ the group defined by the same presentation. Note that $M(3)=M_2$. 

\begin{lemma}
The group $B(m)$ is isomorphic to the dihedral Artin--Tits group~$\mathcal{B}(I_2(m))$. 
\end{lemma}

\begin{proof}
It is straightforward to check that an isomorphism is given by $\rho_1 \mapsto \sigma$, $\rho_2 \mapsto \sigma \tau$. 
\end{proof}

Note that $M(m)$ is cancellative, as divisibility is Noetherian (since the defining relation is homogeneous with $\lambda(\rho_1)=1$ and $\lambda(\rho_2)=2$) and $M(m)$ is generated by two elements $\rho_1$, $\rho_2$ with a single relation of the form $\rho_1 \cdots = \rho_2 \cdots$, hence the defining presentation is right-complemented and the $\theta$-cube condition (Definition~\ref{def_cube}) is vacuously true for triples of distinct generators.  

Setting $\Delta:=\rho_2^{m}$, the following Lemma is the analogue for $M(m)$ of Proposition~\ref{garside_divisors} established in the case of $M_n$:

\begin{lemma}\label{garside_dihedral}
The following holds in $M(m)$: 
\begin{enumerate}
\item We have $(\rho_1 \rho_2^{(m-1)/2})^2=(\rho_2^{(m-1)/2} \rho_1)^2=\Delta$.
\item Let $a_1:= \rho_2^{(m-1)/2} \rho_1 \rho_2^{(m-1)/2}$. Then $\rho_1 a_1=a_1 \rho_1=\Delta$. In particular, both generators $\rho_1$ and $\rho_2$ are are left- and right-divisors of $\Delta$ (and the left- and right-complements of a given generator coincide). 
\item Let $a,b\in M(m)$ such that $ab=\Delta$. Then $ba=\Delta$. 
\end{enumerate}
\end{lemma}

\begin{proof}
The first claim is an immediate consequence of the defining relation of $M(m)$. The second claim follows immediately from the first one. The last claim is a consequence of the cancellativity of $M(m)$ and the second claim, as the claimed property holds for $\rho_1$ and $\rho_2$ (recall that $\Delta$ is a power of $\rho_2$), which generate $M(m)$. 
\end{proof}

\begin{prop}
The pair $(M(m), \Delta)$ is a Garside monoid. The corresponding Garside group is $B(m)$. 
\end{prop}

\begin{proof}
The proof is exactly the same as for $G(M_n)$ (Theorem~\ref{garside_gn}): as noted above, the divisibility in $M(m)$ is Noetherian and the $\theta$-cube condition is vacuously true, hence we have cancellativity and the existence of conditional lcm's in $M(m)$. By Lemma~\ref{garside_dihedral} above, the element~$\Delta$ is a Garside element in $M(m)$, and we then conclude the proof by applying the same arguments as for $G(M_n)$. 
\end{proof}

Of course, adding the relation~$\rho_1^2=1$ to the presentation of $B(m)$ yields a presentation of the dihedral group~$I_2(m)$, as there is only one conjugacy class of reflections in $I_2(m)$.  

\begin{rmq}
The dihedral Artin--Tits groups of even type do not seem to admit a similar description. Indeed, let $B=\mathcal{B}(I_2(4))=\mathcal{B}(B_2)$ be the Artin--Tits group of type~$B_2$, with standard generators~$\sigma_1, \sigma_2$ and braid relation~$\sigma_1 \sigma_2 \sigma_1 \sigma_2=\sigma_2 \sigma_1 \sigma_2 \sigma_1$. Then setting $\rho_1=\sigma_1$, $\rho_2=\sigma_1 \sigma_2$, we get a presentation for $B$ by taking this new set of generators and the relation $\rho_1 \rho_2^2=\rho_2^2 \rho_1$. This appears to us as the natural analogue of the presentations considered in the odd case but in the present case, the monoid generated by $\rho_1$ and $\rho_2$ and subject to the above relation is \textit{not} a Garside monoid: indeed, if it was, then the Garside element $\Delta$ would have a power which is central. Since the center of $B$ is infinite cyclic generated by $(\sigma_1 \sigma_2)^2=\rho_2^2$, it is clear from the above defining relation that $\Delta$ itself would have to be a power of $\rho_2$ as $\rho_1$'s cannot be eliminated using the unique defining relation, say $\Delta=\rho_2^m$. But then $\rho_1$ could not divide $\Delta$ as no relation can be applied to the word $\rho_2^m$, a contradiction. 
\end{rmq}

As a concluding remark, let us note the following. We introduced several monoids in this paper, which either are Garside monoids (like $M_n$ and $M(m)$), or closely related to a Garside monoid (like $\mathcal{H}_n^+$). All of them are defined by the same kind of presentations. The corresponding groups of fractions are braid groups of real or complex reflection groups in several cases, and presentations for these reflection groups can be naturally derived from those of the corresponding monoids (as done in Corollaries~\ref{cor_g12}, \ref{prop_sym_pres} and Remark~\ref{rmq_complex_interval}). This covers the following cases: $G_4$, $G_{8}$, $G_{16}$, $G_{12}$, $\mathfrak{S}_n$ for all~$n$, and $I_2(m)$ for odd~$m$. All these groups have a single conjugacy class of reflections, while the dihedral groups of even type like~$I_2(4)$, for which the above remark shows that there does not seem to exist a Garside monoid similar to the ones introduced in this paper, have two conjugacy classes of reflections. While we do not have any general statement at the moment, it would be interesting to investigate whether reflection groups with a single conjugacy class of reflections, and their braid groups, admit presentations and monoids similar to those introduced in this work.

\end{document}